\documentclass[12pt,a4paper]{article}
\usepackage{a4wide,authblk }
\usepackage{theorem}
\usepackage{amsmath,amssymb,calc}
\usepackage[english]{babel}
\usepackage{graphicx,array}
\usepackage[numbers]{natbib}
\usepackage{tikz}
\usepackage{color} 
\usepackage{enumerate}
\usepackage[toc,page]{appendix}
\usepackage[charter]{mathdesign}
\usepackage{algorithm}
\usepackage{algpseudocode}
\usepackage{tikz,lipsum}
\usepackage[margin=1 cm]{caption}

\newtheorem{theorem}{Theorem}
\newtheorem{assumption}[theorem]{Assumption}

\newtheorem{remark}[theorem]{Remark}
\newtheorem{proposition}[theorem]{Proposition}
\newtheorem{corollary}[theorem]{Corollary}

\newenvironment{proof}{\par\noindent\emph{Proof. }}{\hfill$\square$\par}
\numberwithin{equation}{section}
\providecommand{\href}[2]{#2}

\newcommand{\beq}{\begin{equation}}
	\newcommand{\eq}{\end{equation}}
\newcommand{\E}{\mathbb{E}}

\newcommand{\e}{{\rm e}}

\newcommand{\car}[3][0]{
	\begin{scope}[shift={#2}, rotate=#1, scale=0.04]
		\draw[#3,fill=#3](69.5,15.5)--(69.5,16.8)--(69.3,19.3)--(68.7,21.8)--(67.2,24.3)--(64.9,26.6)--(61.9,28.1)--(58.4,29.2)--(55.,29.4)--(52.7,29.2)--(50.6,28.8)--(43.1,28.9)--(41.9,30.9)--(40.5,31.)--(39.3,31.)--(39.5,29.)--(28.5,28.7)--(18.4,29.8)--(13.7,29.8)--(10.,29.5)--(2.8,27.9)--(1.7,26.9)--(0.4,23.)--(0.,15.5)--(0.,15.5)--(0.4,8.)--(1.7,4.1)--(2.8,3.1)--(10.,1.5)--(13.7,1.2)--(18.4,1.2)--(28.5,2.3)--(39.5,2.)--(39.3,0.)--(40.5,0.)--(41.9,0.1)--(43.1,2.1)--(50.6,2.2)--(52.7,1.8)--(55.,1.6)--(58.4,1.8)--(61.9,2.9)--(64.9,4.4)--(67.2,6.7)--(68.7,9.2)--(69.3,11.7)--(69.5,14.2)--(69.5,15.5);
		\draw[white,fill=white](53.,15.5)--(52.8,18.1)--(52.,20.4)--(51.,22.2)--(49.8,23.8)--(48.6,24.9)--(47.3,25.1)--(45.7,24.9)--(43.,24.6)--(40.2,23.6)--(39.,22.7)--(38.5,21.1)--(38.7,15.5)--(38.7,15.5)--(38.5,9.9)--(39.,8.3)--(40.2,7.4)--(43.,6.4)--(45.7,6.1)--(47.3,5.9)--(48.6,6.1)--(49.8,7.2)--(51.,8.8)--(52.,10.6)--(52.8,12.9)--(53.,15.5);
		\draw[white,fill=white](17.6,15.5)--(17.7,20.4)--(17.5,21.7)--(16.6,22.9)--(15.3,23.4)--(8.9,22.2)--(7.7,20.9)--(6.9,19.8)--(6.3,18.3)--(5.9,15.5)--(5.9,15.5)--(6.3,12.7)--(6.9,11.2)--(7.7,10.1)--(8.9,8.8)--(15.3,7.6)--(16.6,8.1)--(17.5,9.3)--(17.7,10.6)--(17.6,15.5);
		\draw[white,fill=white](64.5,25.3)--(65.3,24.3)--(66.4,23.2)--(67.1,22.1)--(67.5,20.8)--(66.2,20.8)--(65.2,21.5)--(64.2,22.8)--(63.7,24.2)--(63.5,26.)--(64.5,25.3);
		\draw[white,fill=white](63.5,5.7)--(62.5,5.)--(62.7,6.8)--(63.2,8.2)--(64.2,9.5)--(65.2,10.2)--(66.5,10.2)--(66.1,8.9)--(65.4,7.8)--(64.3,6.7)--(63.5,5.7);
	\end{scope}
}

\title{Optimal capacity allocation for heavy-traffic fixed-cycle traffic-light queues and intersections}

\author[1]{Marko Boon}
\author[1]{Guido Janssen}
\author[1,2]{Johan van Leeuwaarden}
\author[1]{Rik Timmerman}
\affil[1]{Eindhoven University of Technology}
\affil[2]{Tilburg University}

\providecommand{\keywords}[1]
{
	\small	
	\textbf{\textit{Keywords---}} #1
}

\begin{document}






\maketitle	
\abstract{Setting traffic light signals is a classical topic in traffic engineering, and important in heavy-traffic conditions when green times become scarce and longer queues are inevitably formed. For the fixed-cycle traffic-light queue, an elementary queueing model for one traffic light with cyclic signaling, we obtain heavy-traffic limits that capture the long-term queue behavior. We leverage the limit theorems to obtain sharp performance approximations for one queue in heavy traffic. We also consider optimization problems that aim for optimal division of green times among multiple conflicting traffic streams. We show that inserting heavy-traffic approximations leads to tractable optimization problems and close-to-optimal signal prescriptions. }%


\vspace{1cm}




%


\section{Introduction} \label{sec:intro}
We consider the problem of finding an optimal allocation of green times for conflicting traffic streams that arrive at the same intersection. This design problem is mathematically challenging due to the stochastic nature of traffic and the dimensionality that comes with multiple flows. We reduce the complexity of the problem by considering cyclic signaling, so that in each fixed cycle all streams receive some share of the available cycle time. This modeling assumption makes available the so-called fixed-cycle traffic-light (FCTL) queue, a classical queueing model for static and cyclic signaling~\cite{boon2019pollaczek,boon2018networks,darroch1964traffic,hagen1989comparison,oblakova2019exact,pacheco2017queues,van2006delay}

Optimizing traffic light settings is particularly relevant when the vehicle-to-capacity ratio approaches the maximal sustainable level. To deal with such scenarios, we establish heavy-traffic limit theorems for the FCTL queue that provide sharp performance approximations for \emph{one} queue in heavy traffic. We use these heavy-traffic approximations to approximatively solve the optimization problems that aim for optimal division of green times among multiple conflicting traffic streams. It turns out that the reduced complexity of the heavy-traffic approximations leads to more tractable optimization problems and close-to-optimal signal prescriptions. Our optimization problems are reminiscent of the so-called capacity allocation problem, originally formulated by Kleinrock in~\cite{kleinrock1976queueing} for dividing capacity among multiple independent $M/M/1$ queues, with the aim of minimizing the average waiting time in all queues. This optimization problem has an elegant explicit solution, and was later generalized by Wein in~\cite{wein1989capacity} for Jackson networks with product-form solutions. \cite{wein1989capacity}  solved the optimization problem by relaxing the original problem through insertion of classical heavy-traffic approximations. We adopt a similar approach, but need to deal with the specific challenges that come with considering FCTL queues rather than standard queues. Optimal green-time allocations have been considered before, e.g. by Webster~\cite{webster1958traffic} and several other studies that built on the results of Webster.


The heavy-traffic scenario that we consider lets the cycle length grow large while at the same time the load or vehicle-to-capacity ratio approaches 100\%. As far as we are aware, this is the first study that applies this scenario for the FCTL queue. Related scalings in continuous-time single-server queues are referred to as ``nearly-deterministic regime''~\cite{sigman2011heavyAP} and \cite{sigman2011heavyQS} and in multi-server settings as the Halfin-Whitt regime or Quality-and-Efficiency-Driven (QED) regime~\cite{halfin1981heavy} and \cite{van2019economies}. The term QED regime was coined because queueing systems in this regime can deal with high vehicle-to-capacity ratios while the probability of no delay stays strictly between 0 and 1. We show that similar favorable properties exist for the heavy-traffic FCTL queue. It seems that vehicle-actuated signaling strategies also exhibit the same favorable properties as is shown in~\cite{timmerman2020new} by means of simulation. The latter result points out that the framework in this paper is applicable for traffic signals in general and is not only restricted to static signaling.

To establish the FCTL heavy-traffic results, we use the transform expressions obtained in~\cite{boon2019pollaczek}. In particular, we use the contour-integral representation for the probability generating function (PGF) of the overflow queue. From this PGF, essentially all relevant information about the stationary behavior of the FCTL can be obtained, by taking derivatives at one for the moments, derivatives at zero for the distribution, and by using simple recursions to obtain the queue lengths at all moments within the cycle and the stationary delay distribution. Establishing scaling limits requires showing convergence of transforms, and proves more challenging. The main idea of our proof is to expand the integrand of the contour integral, and to show that in the heavy-traffic regime, only the first few terms of the expansion (up to leading order) dominate the numerical value of the integral. Making such observations rigorous, however, requires careful analysis. While this analysis is new, in classical queueing theory, establishing heavy-traffic results through the asymptotic evaluation of contour integrals was done  in e.g.~\cite{cohen1982single} and \cite{kingman1962queues} for single-server queues and in~\cite{janssen2015novel} for classical bulk-service queues.

In the heavy-traffic regime we consider, the scaled queue length turns out to converge to a reflected Gaussian random walk, a stochastic process that occurs in a range of other applications and that has been studied in great detail~\cite{blanchet2006complete,chang1997ladder,janssen2007lerch}. We exploit this connection to convert known results for the reflected Gaussian random walk into heavy-traffic approximations for the FCTL queue. These heavy-traffic approximations are considerably easier than the exact (contour-integral) expressions, which presents analytic advantages when considering the optimization problem of finding the optimal traffic light setting for intersections with cyclic arrangements of multiple conflicting traffic flows. The heavy-traffic approximations let us obtain closed-form expressions for such optimal settings.

Our main contributions can be summarized as follows:
\begin{itemize}
\item[(i)]   For the FCTL queue we obtain novel heavy-traffic limit theorems by asymptotic evaluation of contour integrals, showing that the scaled queue length converges to the reflected Gaussian random walk.
\item[(ii)]  We leverage the limit theorems to obtain sharp performance approximations for one queue in heavy traffic, utilizing existing results for the reflected Gaussian random walk.
\item[(iii)] We consider optimization problems that find the optimal division of green times among multiple conflicting traffic streams, and show that inserting heavy-traffic approximations leads to tractable optimization problems and close-to-optimal signal prescriptions.
\end{itemize}

The paper is organized as follows. In Section~\ref{sec:ht_fctl} we present the heavy-traffic analysis of the  FCTL queue.  By using the resulting heavy-traffic approximations, we present in Section~\ref{sec:opt_scaling} the optimal traffic-light settings for the situation of multiple conflicting traffic streams. Numerical examples are presented in Section~\ref{sec:numerics}. We present the main heavy-traffic proof in Section~\ref{sec:proofFCTLconvergence} and
conclude with a summary and topics for further research in Section~\ref{sec:con}. Remaining proofs are deferred to Appendix~\ref{sec:remaining}.

\section{FCTL queue in heavy traffic} \label{sec:ht_fctl}

The FCTL queue is a discrete-time queueing model aimed to capture the queueing dynamics at traffic lights at an intersection. 
A lane that is connected to the intersection has a dedicated traffic signal, alternating between red and green (the amber period is typically incorporated in one of these two periods). We divide time into slots of unit length, which can be interpreted as the interdeparture time of queued vehicles. From the viewpoint of one specific lane, the cycle is divided into a green period of $g$ time slots (allowing at most $g$ delayed vehicles to depart) and a red period of length $r$. The total cycle length is thus $c=g+r$ and is fixed, although the $g$ and $r$ for individual lanes may be different. Moreover, a green period for one lane implies a red period for all other lanes as mentioned before. During any of the green slots at a specific lane, the following may happen: if there is a queue of vehicles, exactly one vehicle may depart from the queue during that slot; if there is \emph{no} queue at the start of the slot, \emph{all} arriving vehicles during that slot may pass at full speed without any delay, which is usually referred to as the FCTL assumption. During a slot in the red period, all arriving vehicles join the queue in front of the traffic light. We further assume that the arrivals during each slot are independent and identically distributed (i.i.d.) and we denote their distribution with $Y$. We denote the mean number of arrivals during a slot with $\mu$ and its standard deviation with $\sigma$. We define $\rho$ to be the vehicle-to-capacity ratio or the saturation level of a lane, which satisfies $\rho = \mu c / g$.

The queue length process at the end of the green period gives rise to a Lindley-type recursion that we know from the basic theory on the single-server $GI/G/1$ queue. To see this, denote with $X_{g}^{(k)}$ the queue length at the end of the green period in cycle $k$ and with $Y_{k}$ the number of arriving delayed vehicles in between the end of the $k^\textrm{th}$ and the end of the $(k+1)^\textrm{th}$ green period. This gives the recursion, see e.g.~\cite{boon2019pollaczek},
\begin{equation}\label{eq:lindley}
	X_{g}^{(k+1)} = \max\{0, X_{g}^{(k)} + Y_k - g\}.
\end{equation}
Observe that \eqref{eq:lindley} is not a standard Lindley recursion, due to the FCTL assumption and hence the intricate dependency between the \emph{delayed} arrivals $Y_k$ and $X_g^{(k)}$.

We shall focus on the limiting queue length $X_g:=\lim_{k\to\infty}X_g^{(k)}$, which is well defined assuming $\rho<1$. We refer to $X_g$ as the \emph{overflow queue}. The probability generating function of $X_g$ was first obtained in~\cite{darroch1964traffic}, and recently an alternative expression has been derived in \cite[Theorem 1]{boon2019pollaczek}. The latter expression allows us to establish a heavy-traffic limit theorem for the overflow queue.
We consider a heavy-traffic regime that  connects  the cycle length and the green period
according to
\begin{equation}\label{eq:ht_scaling}
	g = \mu c +\beta \sigma \sqrt{c}.
\end{equation}
Here $\beta>0$ is a parameter that can be chosen freely,
and optimal choices for $\beta$ will be obtained in Section~\ref{sec:opt_scaling}.
The main intuition for considering the regime in Equation~\eqref{eq:ht_scaling} is as follows. In heavy traffic there will be many delayed cars, and during each cycle $g$
 delayed cars can depart while
on average $\mu c$ new delayed cars will arrive.
 We therefore choose the green period as roughly $c\mu$, but add $\beta \sigma \sqrt{c}$ to account for variability of the number of newly arriving cars. Observe that for large cycles, $\beta \sigma \sqrt{c}$ will be considerably smaller than $c\mu$. In the heavy-traffic regime we consider, $c$ will be large, $c\mu$ is the dominant term strictly required for stability, while $\beta \sigma \sqrt{c}$ is a hedge against uncertainty. To understand the effect of this hedge,  substitute \eqref{eq:ht_scaling} into \eqref{eq:lindley} to obtain $X_{g}^{(k+1)}  = \max\{0,X_{g}^{(k)}+Y_k - \mu c -\beta\sigma\sqrt{c}\}$. After dividing the term $Y_k-\mu c-\beta\sigma\sqrt{c}$ by the standard deviation of the number of arrivals per cycle, $\sigma \sqrt{c}$, we expect it to be approximately normally distributed (with mean $-\beta$ and unit variance) when $c$ grows large because of the Central Limit Theorem (CLT). This is not entirely straightforward, because $Y_k$ cannot be interpreted as the sum of $c$ independent random variables, and hence the CLT cannot be applied directly. We therefore resort to the transform method. We take the expression for the
 probability generating function of $X_g$ established in~\cite{boon2019pollaczek}, and show that this transform converges in the heavy-traffic regime \eqref{eq:ht_scaling} with $c\to\infty$ to the transform of a non-degenerate random variable
$M_\beta$. The convergence of transforms then implies the convergence of the underlying random variables. Here, $M_\beta$ is a special random variable equal to the all-time maximum of the so-called Gaussian random walk with drift $-\beta$ and variance 1, see e.g.~\cite{janssen2007lerch} for a detailed study of various characteristics of $M_\beta$, including expression and approximations for all moments. We will give more details on $M_\beta$ later, and first present our main heavy-traffic limit theorem. Let $\stackrel{d}{\to}$ denote convergence in distribution.



\begin{theorem}[Heavy-traffic limit theorem]\label{thm:ht_FCTL}
	Assume that $\mathbb{E}[z^Y]$ is analytic within a disk of radius $R$ with $R>1$, and $\mu<1$, $\sigma^2>0$. Under scaling \eqref{eq:ht_scaling}, as $c\to\infty$,
	\begin{equation}\label{eq:conv_Xg}
		\frac{1}{\sigma\sqrt{c}}X_{g}\stackrel{d}{\to} M_\beta,
	\end{equation}
		\begin{equation}\label{eq:conv_prob}
		\mathbb{P}\left(\frac{X_g}{\sigma\sqrt{c}}=0\right) = \mathbb{P}(M_\beta=0)\left(1+O\left(\frac{1}{\sqrt{c}}\right)\right)
	\end{equation}
	and for $k\geq 1$,
	\begin{equation}\label{eq:conv_moments}
		\mathbb{E}[X_{g}^k] = \left(\sigma\sqrt{c}\right)^k\mathbb{E}[M_{\beta}^k]\left(1+O\left(\frac{1}{\sqrt{c}}\right)\right).
	\end{equation}
\end{theorem}

The proof is deferred to Section~\ref{sec:proofFCTLconvergence}. Theorem~\ref{thm:ht_FCTL}  has two practical implications. First, since the scaled overflow queue $X_g$ converges to a non-degenerate limiting variable, the scaling rule \eqref{eq:ht_scaling} can serve as a guiding principle for choosing the cycle length as a function of traffic pressure. That is, since there exists a non-degenerate limit, scaling rules that let $g$ scale faster (e.g.~$g = \mu c +\beta \sigma {c}^{2/3}$ or $g = (\mu+\beta) c$) or slower (e.g.~$g = \mu c +\beta \sigma {c}^{1/3}$ or $g = \mu c+\beta$) likely lead to degenerate behavior in the large cycle limit $c\to\infty$, that is $X_g$ converges with high probability to either $0$ or $\infty$.
The second practical implication is that known results for the limit $M_\beta$ can be converted into approximations for $X_g$. As Theorem~\ref{thm:ht_FCTL} suggests, for large enough $c$,
\begin{align}
&\mathbb{E}[X_{g}] \approx \sigma\sqrt{c}\,
\mathbb{E}[M_\beta],\quad
\mathbb{P}(X_g=0)\approx \mathbb{P}(M_\beta=0).
\end{align}
 Let $\zeta(.)$ denote the Riemann zeta function. For $0<\beta<2\sqrt{\pi}$ it was shown in ~\cite{janssen2007lerch} that
\begin{align}
\label{eq:approxEM}
&\mathbb{E}[{M}_\beta]= \frac{1}{2\beta}+\frac{\zeta(1/2)}{\sqrt{2\pi}}+\frac{\beta}{4}+\frac{\beta^2}{\sqrt{2\pi}}\sum_{r=0}^{\infty}\frac{\zeta(-1/2-r)}{r!(2r+1)(2r+2)}\left(\frac{-\beta^2}{2 }\right)^r,\\
&\mathbb{P}(M_\beta=0)= \sqrt{2}\beta\exp\left\{\frac{\beta}{\sqrt{2\pi}}\sum_{r=0}^{\infty}\frac{\zeta(1/2-r)}{r!(2r+1)}\left(\frac{-\beta^2}{2}\right)^r\right\}.\label{eq:approxP0}
\end{align}
These expressions give heavy-traffic approximations for the overflow queue that are accurate when $\beta$ is small and $c$ is sufficiently large. Expression~\eqref{eq:approxEM} also reveals that for small $\beta$,
 $\mathbb{E}[M_\beta] \approx 1/(2\beta)$, a particularly easy approximation that will be helpful when we optimize signal settings later in the paper.

We also derive other approximations for $\mathbb{E}[X_g]$ and $\mathbb{P}(X_g=0)$ that are more accurate, in particular for smaller $c$ and larger $\beta$. Let us introduce the integrals
\begin{align}
G_0(b) & =\int_{0}^\infty \frac{t^2}{b^2+t^2}\frac{e^{-b^2-t^2}}{1-e^{-b^2-t^2}}\mathrm{d}t ,\label{eq:G0}\\
G_1(b) & =\int_{0}^\infty \frac{e^{-b^2-t^2}}{1-e^{-b^2-t^2}}\mathrm{d}t,
\end{align}
that can be computed numerically by standard software packages. In addition,~\cite{janssen2015novel}, Equations (4.27), (4.29) and (4.31), provide $\zeta$-series such as in Equations~\eqref{eq:approxEM} and \eqref{eq:approxP0}, for $G_0(b)$ and $G_1(b)$ as well as rapidly convergent series involving the standard Gaussian and the complementary error function. One consequence of the results that is of the latter type is the series representation
\begin{equation}\label{eq:G0prime(b)}
    G_{0}^\prime(b) = -\sqrt{\pi} \sum_{k=0}^\infty \int_{b\sqrt{k+1}}^\infty e^{-t^2} \mathrm{d}t
\end{equation}
that shows that $G_{0}^\prime(b)$ is negative and strictly increasing in $b>0$, which will be used later on.
We prove the following result in Appendix~\ref{sec:remaining}.
\begin{proposition}[Refined heavy-traffic approximations]\label{thm:mean}
The mean overflow queue satisfies, as $c\to\infty$,
\begin{align}
	\mathbb{E}[X_{g}] & = \frac{\sqrt{2}}{\pi}\left(\sigma\sqrt{c}+\frac{\beta\sigma^2}{2\mu}\right)G_0(b(\beta))+\frac{\theta\beta}{\pi}G_{1}\left(\frac{\beta}{\sqrt{2}}\right)+O\left(\frac{1}{\sqrt{c}}\right), \label{eq:int_meansqrtc}
\end{align}
where
\begin{align}
b(\beta) & = \frac{\beta}{\sqrt{2}}\left(1+\frac{\beta\sigma}{\mu\sqrt{c}}\right)^{-1/2},\\
a & = \frac{\mu_3-\mu^3-3(1+\mu)\sigma^2}{\mu},\label{eq:a}\\
\theta & = \frac{\sigma^2}{\mu\sqrt{2}}\left(\frac{\mu}{\sigma^2}+\frac{1}{3}\left(\frac{\mu}{\sigma^2}\right)^2a-1\right)\label{eq:theta}
\end{align}
with $\mu_3$  the third moment of $Y$.
\end{proposition}

A direct consequence of  Proposition~\ref{thm:mean} is the slightly easier approximation
 \begin{align}
	\mathbb{E}[X_{g}] &= \frac{\sqrt{2}}{\pi}\sigma\sqrt{c}\,G_0\left(\frac{\beta}{\sqrt{2}}\right)+O(1).\label{eq:int_mean1}
\end{align}
 Tables~\ref{t:numBeta0.1} and \ref{t:numBeta1} show the asymptotic approximations we have just derived. As expected, the approximations become more accurate for larger $c$. However, the approximations also serve as useful, somewhat looser approximations for small and moderate $c$-values.
In conclusion, we have derived two asymptotic approximations for $\mathbb{E}[X_g]$, the first-order approximation \eqref{eq:int_mean1} with error $O(1)$, and the refined approximation \eqref{eq:int_meansqrtc}  with error $O(1/\sqrt{c})$. Although both approximations perform well, already for small values of $c$, \eqref{eq:int_meansqrtc} is more accurate than \eqref{eq:int_mean1} for values of $\beta$ as large as 2. Both approximations will be employed in the next section for the purpose of solving optimization problems.

\begin{table}[h!]
{\small\begin{center}
\begin{tabular}{rrrrrrrr}\hline
    & &   \multicolumn{3}{c}{$\mathbb{P}(X_{g} = 0)$}       & \multicolumn{3}{c}{$\mathbb{E}[X_{g}]$} \\ \cline{3-4} \cline{6-8}
$g$ & $c$ & true value  & Appr.~\eqref{eq:approxP0}  &  & true value  & Appr.~\eqref{eq:int_mean1}      & Appr.~\eqref{eq:int_meansqrtc}
\\ \hline
$10$    & 32.3  &  0.1649        & 0.1334   & & 13.935 & 13.826  & 13.985 \\
$20$    & 65.2  &  0.1551        & 0.1334   & & 19.767 & 19.644  & 19.803               \\
$30$    & 98.2  &  0.1509        & 0.1334   & & 24.238 & 24.109  & 24.267               \\
$50$    & 164.3 &  0.1468        & 0.1334   & & 31.324 & 31.188  & 31.346               \\
$100$   & 330.0 &  0.1427        & 0.1334   & & 44.340 & 44.198  & 44.356               \\ \hline
\end{tabular}
\end{center}
\caption{\small{Exact results for $\mathbb{P}(X_g=0)$ and $\mathbb{E}[X_g]$ for several values of $g$ and $c$ with Poisson arrivals with mean $0.3$ in each slot and with $\beta=0.1$.}}\label{t:numBeta0.1}}
\end{table}

\begin{table}[h!]
{\small\begin{center}
\begin{tabular}{rrrrrrrr}\hline
    & &   \multicolumn{3}{c}{$\mathbb{P}(X_{g} = 0)$}       & \multicolumn{3}{c}{$\mathbb{E}[X_{g}]$} \\ \cline{3-4} \cline{6-8}
$g$ & $c$ & true value  & Appr.~\eqref{eq:approxP0}  &  & true value  & Appr.~\eqref{eq:int_mean1}      & Appr.~\eqref{eq:int_meansqrtc} \\\hline
$10$    & 24.3  & 0.8450   &  0.8005        & & 0.3944 & 0.3414  & 0.4437              \\
$20$    & 53.3  & 0.8312   &  0.8005        & & 0.5664 & 0.5055  & 0.5996               \\
$30$    & 83.3  & 0.8253   &  0.8005        & & 0.6960 & 0.6319  & 0.7225               \\
$50$    & 144.7 & 0.8200   &  0.8005        & & 0.8998 & 0.8326  & 0.9199               \\
$100$   & 301.6 & 0.8138   &  0.8005        & & 1.2722 & 1.2021  & 1.2860               \\ \hline
\end{tabular}
\end{center}
\caption{\small{Exact results for $\mathbb{P}(X_g=0)$ and $\mathbb{E}[X_g]$ for several values of $g$ and $c$ with Poisson arrivals with mean $0.3$ in each slot and with $\beta=1$.}}\label{t:numBeta1}}
\end{table}

\section{Capacity allocation problems}\label{sec:opt_scaling}

We now turn to optimal green-time allocations for an intersection with  $n$ lanes, where each lane is modeled separately as an FCTL queue. Let $\mu_i$ denote the mean arrival rate at lane $i$ and $g_i$ the green time allocated to lane $i$ within one cycle of length $c$.  While the lanes operate independently once the green times are fixed, they do depend on each other through the cycle time $c$ and the green time of one lane corresponds to a red period for the other lanes. We leverage this independence across lanes and the asymptotic approximations developed in Section~\ref{sec:ht_fctl} to formulate several optimization problems that search for the vector of green times that minimizes the total expected overflow queue.



\subsection{Minimizing the sum of overflows} \label{sec:unweighted}

Consider the problem of finding the green times that minimize the sum of the mean queue lengths at the end of the green periods $\sum_{i=1}^n\mathbb{E}[X_{g_i,i}]$.
Assume that $c$ is fixed, and let $r_T<c$ represent the time that cannot be used as green time. This $r_T$ could model e.g. clearing times between lanes. Hence, $c=r_T+\sum_{i=1}^n g_i$.
Again applying the substitution as in~\eqref{eq:ht_scaling}, $g_i = \mu_i c +\beta_i \sigma_i \sqrt{c}$, for $i=1,\ldots,n$, this gives the following optimization problem:

\begin{equation}\label{eq:prob1}
\begin{aligned}
& \underset{\beta_1,\ldots,\beta_n}{\text{minimize}} && \sum_{i=1}^n \mathbb{E}[X_{g_i,i}]\\
& \text{subject to}&& \sum_{i=1}^n \beta_i\sigma_i\sqrt{c} = c(1- \mu_T)-r_T;\\
& && \beta_i> 0, \; i = 1, \ldots,n,
\end{aligned}
\end{equation}
with $\mu_T=\sum_{i=1}^n \mu_i$. The first constraint in~\eqref{eq:prob1} relates to the requirement $c=r_T+\sum_{i=1}^n g_i$ and together with the constraints $\beta_i>0$ for all $i$, it is ensured that each $g_i$ might be chosen so as to ensure a vehicle-to-capacity ratio less than 1 for each lane as $c(1-\mu_T)-r_T>0$.

Optimization problem~\eqref{eq:prob1} seems mathematically intractable due to the lack of an explicit expression for the objective function $\sum_{i=1}^n \mathbb{E}[X_{g_i,i}]$. We shall therefore use approximations based on Equations~\eqref{eq:int_mean1} and \eqref{eq:int_meansqrtc} to replace the objective function with a heavy-traffic approximation, which then leads to a tractable, more structured optimization problem.


Using \eqref{eq:int_mean1} gives the following optimization problem:
\begin{equation}\label{eq:approach1a}
\begin{aligned}
& \underset{\beta_1,\ldots,\beta_n}{\text{minimize}} && \sum_{i=1}^n \frac{\sigma_i}{\pi} \sqrt{2c}\,G_0(\beta_i/\sqrt{2}) \\
& \text{subject to}&& \sum_{i=1}^n \beta_i\sigma_i\sqrt{c} = c(1- \mu_T)-r_T;\\
& && \beta_i> 0, \; i = 1, \ldots,n.
\end{aligned}
\end{equation}
\begin{proposition}\label{thm:noCosts}
 Optimization problem~\eqref{eq:approach1a} is solved by
	\begin{equation}\label{eq:sol1}
	\beta_i = \frac{c(1-\mu_T)-r_T}{\sqrt{c}\sum_{j=1}^n \sigma_j}=:\beta_{*}.
	\end{equation}
\end{proposition}

\begin{proof}
	Introduce the Lagrange multiplier $\lambda_0\in\mathbb{R}$, so that
	\begin{equation}
		\frac{\partial}{\partial \beta_i}\left(\sum_{j=1}^n \frac{\sigma_j\sqrt{2c}}{\pi} G_0(\beta_j/\sqrt{2})\right)=\lambda_0 \frac{\partial}{\partial \beta_i}\left(\sum_{j=1}^n \beta_j\sigma_j\sqrt{c} - c(1- \mu_T)+r_T\right),
	\end{equation}
	for $i=1,...,n$.
	This gives
	\begin{equation}\label{eq:lagrange}
		G_{0}^\prime (\beta_i/\sqrt{2})=\pi\lambda_0.
	\end{equation}
	The function $G_{0}^\prime(b)$ is negative and strictly increasing in $b>0$, see Equation~\eqref{eq:G0prime(b)}. Combining this with the fact that  $\lambda_0$ is independent of the index $i$, we conclude that the $\beta_i$ are the same for $i=1,...,n$ 
	and should satisfy
	\begin{align}
		& \beta_i \sqrt{c}\sum_{j=1}^n \sigma_j = c(1- \mu_T)-r_T,
	\end{align}
	for $i=1,...,n$, which completes the proof.
\end{proof}

	Proposition~\ref{thm:noCosts} shows that the optimal parameters $\beta_i$ should be equal for all lanes. 
We now turn to the second approximation for the problem formulated in Equation~\eqref{eq:prob1}, based on the refined heavy-traffic approximation in \eqref{eq:int_meansqrtc}:
\begin{equation}\label{eq:approach1b}
\begin{aligned}
& \underset{\beta_1,\ldots,\beta_n}{\text{minimize}} &&\sum_{i=1}^n \frac{\sqrt{2}}{\pi}\left(\sigma_i\sqrt{c}+\frac{\beta_i\sigma_{i}^2}{2\mu_i}\right)G_0(b_i(\beta_i))+\frac{\theta_i\beta_i}{\pi}G_{1}\left(\frac{\beta_i}{\sqrt{2}}\right) \\
& \text{subject to}&& \sum_{i=1}^n \beta_i\sigma_i\sqrt{c} = c(1- \mu_T)-r_T;\\
& && \beta_i> 0, \; i = 1, \ldots n.
\end{aligned}
\end{equation}
\begin{theorem}\label{thm:noCostsHigherOrder}
	Optimization problem~\eqref{eq:approach1b} is solved by
	\begin{equation}\label{dimrule}
		\beta_i=\beta_* + \Omega_i(\beta_*),	\quad i=1,\ldots,n,
	\end{equation}
	with $\beta_*$ as in \eqref{eq:sol1},
		\begin{equation}
		\Omega_i(\beta_*)= \sqrt{\frac{2}{c}}\frac{1}{G_{0}^{\prime\prime}\left(\beta_*/\sqrt{2}\right)}\left(\frac{\sum_{j=1}^n K_{j}}{\sum_{j=1}^n\sigma_j}-\frac{K_{i}}{\sigma_i}\right),
	\end{equation}
	 and
	\begin{align}
		K_{i} = & \frac{\sigma_{i}^2}{\sqrt{2}\mu_i}G_0\left(\frac{\beta_*}{\sqrt{2}}\right)-\frac{\beta_*\sigma_{i}^2}{2\mu_i}G_{0}^\prime\left(\frac{\beta_*}{\sqrt{2}}\right)-\frac{\beta_{*}^2\sigma_{i}^2}{2\sqrt{2}\mu_i}G_{0}^{\prime\prime}\left(\frac{\beta_*}{\sqrt{2}}\right)+\\&\theta_i G_{1}\left(\frac{\beta_*}{\sqrt{2}}\right)+\frac{\theta_i\beta_*}{\sqrt{2}}G_{1}^\prime\left(\frac{\beta_*}{\sqrt{2}}\right).\label{eq:k0i}
	\end{align}
\end{theorem}

The proof of Theorem~\ref{thm:noCostsHigherOrder} is presented in Appendix~\ref{sec:remaining}. The result may seem complicated at first glance, but in fact reveals a remarkably elegant structure. The term $\Omega_i(\beta_*)$ can be thought of as a refinement of $\beta_*$, due to using the refined approximation \eqref{eq:int_meansqrtc} instead of \eqref{eq:int_mean1}.
An intriguing finding is that $\Omega_i(\beta_*)$ can be written explicitly in terms of $\beta_*$.  From that perspective, the rule in \eqref{dimrule} can be interpreted as a two-step procedure. First divide the green time into parts of length $
	g_i = \mu_i c +\beta_* \sigma_i \sqrt{c},
$
and then correct or refine this fair division using  $\beta_*+\Omega_i(\beta_*)$ instead of $\beta_*$.
Note that in this second step, lane $i$ gets a larger or smaller share depending on the sign of
	\begin{equation}
	\frac{\sum_{j=1}^n K_{j}}{\sum_{j=1}^n\sigma_j}-\frac{K_{i}}{\sigma_i}.
	\end{equation}

Generally, the solution to optimization problem \eqref{eq:approach1b} will lead to more accurate results for the minimization of $\sum_{i=1}^n \mathbb{E}[X_{g_i,i}]$ than the solution to the optimization problem \eqref{eq:approach1a}, as the approximation of the individual $\mathbb{E}[X_{g_i,i}]$ terms is more accurate. We return to this observation in Section~\ref{sec:numerics}, Example~1.

\begin{remark}\label{rem:nonintegeroptimal}
	The solutions of both optimization problems formulated above generally result in non-integer values for $g_i$. Depending on the exact setting, we might opt for rounding the values to the nearest integer or rounding the value of $g_i$ down (along with checking for stability). An alternative procedure is to allow for a random green time. If $G_i$ denotes such a random green time, we can choose it in the following way: $G_i$ is equal to $\lfloor g_i \rfloor$ with probability $p$ and equal to $\lceil g_i\rceil$ with probability $1-p$ such that $g_i = p\lfloor g_i \rfloor+(1-p)\lceil g_i\rceil$. We show how this can be accounted for in Remark~\ref{rem:noninteger}.
\end{remark}

\subsection{Minimizing the weighted sum of overflows}\label{sec:weighted}

In practice, it might be preferred to give priority to certain lanes. This might be modeled by introducing weights associated with each lane. Therefore, we assume next that lane $i$ gets weight $d_i>0$ and formulate the optimization problem
\begin{equation}\label{eq:prob2}
\begin{aligned}
& \underset{\beta_1,\ldots,\beta_n}{\text{minimize}} && \sum_{i=1}^n d_i\mathbb{E}[X_{g_i,i}]\\
& \text{subject to}&& \sum_{i=1}^n \beta_i\sigma_i\sqrt{c} = c(1-\mu_T)-r_T;\\
& && \beta_i> 0, \; i = 1, \ldots,n.
\end{aligned}
\end{equation}	
Due to the weights $d_1,\ldots,d_n$ we cannot (approximately) solve the problem~\eqref{eq:prob1} with the heavy-traffic approximations in \eqref{eq:approach1a} and \eqref{eq:approach1b}.
We therefore resort to the approximation $\mathbb{E}[X_{g_i,i}]\approx \sigma_i\sqrt{c}/(2\beta_i)$ as derived in \eqref{eq:approxEM}, and solve the problem
\begin{equation}\label{eq:approach2a}
\begin{aligned}
& \underset{\beta_1,\ldots,\beta_n}{\text{minimize}} && \sum_{i=1}^n d_i\frac{\sqrt{c}\sigma_i}{2\beta_i}\\
& \text{subject to}&& \sum_{i=1}^n \beta_i\sigma_i\sqrt{c} = c(1-\mu_T)-r_T;\\
& && \beta_i > 0, \; i = 1, \ldots,n.
\end{aligned}
\end{equation}	

\begin{proposition}\label{thm:costsExplicit}
	Optimization problem~\eqref{eq:approach2a} is solved by
	\begin{equation}\label{eq:sol2}
	\beta_i = \frac{\sqrt{d_i}(c(1-\mu_T)-r_T)}{\sqrt{c}\sum_{j=1}^n\sqrt{d_j}\sigma_j}.
	\end{equation}
\end{proposition}

\begin{proof}
Follows from the same Lagrange multiplier technique as in the proof of Proposition~\ref{thm:noCosts}.
\end{proof}

Notice that Equation~\eqref{eq:sol2} reduces to  Equation~\eqref{eq:sol1} for  $d_i=1$.
We next use a more accurate approximation for the $\mathbb{E}[X_{g_i,i}]$ and define the following minimization problem.
\begin{equation}\label{eq:approach2b}
\begin{aligned}
& \underset{\beta_1,\ldots,\beta_n}{\text{minimize}} && \sum_{i=1}^n d_i \frac{\sigma_i}{\pi}\sqrt{2c}\,G_0(\beta_i/\sqrt{2})\\
& \text{subject to}&& \sum_{i=1}^n \beta_i\sigma_i\sqrt{c} = c(1-\mu_T)-r_T;\\
& && \beta_i> 0, \; i = 1, \ldots,n.
\end{aligned}
\end{equation}

\begin{corollary} \label{cor:weigthedNumerical}
	There exists a unique solution to optimization problem~\eqref{eq:approach2b}.\label{thm:numericalMinSum}
\end{corollary}

\begin{proof}
Along the same lines as in the proof of Proposition~\ref{thm:noCosts}, we get that there exists a Lagrange multiplier $\lambda_0\in\mathbb{R}$ satisfying
\begin{equation}\label{eq:proofCor_exp_beta}
G_{0}^\prime (\beta_j/\sqrt{2})=\frac{\pi\lambda_0}{d_j\sqrt{c}}.
\end{equation}
As $G_{0}^\prime(b)$ is a strictly increasing function in $b$, it is invertible and thus Equation~\eqref{eq:proofCor_exp_beta} can be solved for $\beta_j$. This implies that a Lagrange multiplier $\lambda_0$ exists and that the problem formulated in Equation~\eqref{eq:approach2b} is solvable.
\end{proof}

While the minimization problem in \eqref{eq:approach2b} cannot be solved analytically, Corollary~\ref{cor:weigthedNumerical} implies that a numerical solution can readily be found. The optimization problem as formulated in Equation~\eqref{eq:approach2b} is a constrained convex optimization problem for which standard numerical solvers exist. One could, e.g., use Interior Point methods to find the optimal $\beta_{i}$'s.

\section{Numerical examples of capacity allocation}\label{sec:numerics}

We now demonstrate the capacity allocation procedures developed in Section~\ref{sec:opt_scaling}, that in turn also use the asymptotic approximations for the mean overflow established in Section~\ref{sec:ht_fctl}.
In particular, the first-order approximation  \eqref{eq:int_mean1} and the refined approximation \eqref{eq:int_meansqrtc}  were both used to solve capacity allocation problems in the asymptotic regime where cycle times become large. This led to asymptotic dimensioning rules that prescribe how to divide the cycle time over the various lanes, and in particular how to choose the green time in an (asymptotically) optimal manner. Because the capacity allocation problems in Section~\ref{sec:opt_scaling} were solved analytically, we have conducted many numerical experiments for assessing the effectiveness of the asymptotic results, for various cycle lengths and distributional assumptions on the arrival processes. From this, we concluded that the asymptotic dimensioning rules perform well, also for settings with a small or moderate cycle length and/or relatively small vehicle-to-capacity ratios. We shall now substantiate these findings by discussing two examples in more detail.

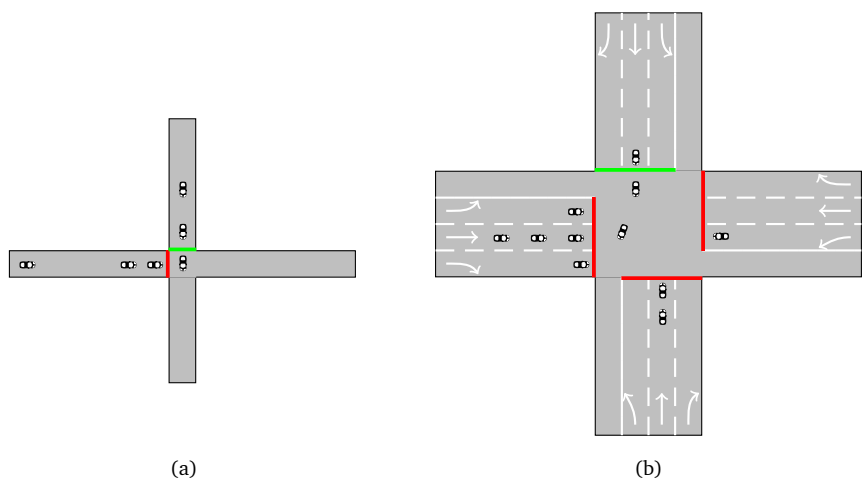
\begin{figure}[ht]
\centering
\begin{tabular}{cc}
\begin{tikzpicture}[scale=0.07]
\draw[white,fill=white](30,10) rectangle (35,0);
\draw[black,fill=lightgray](30,10) rectangle (35,60);

\draw[black,fill=lightgray](0,30) rectangle (65,35);
\draw[lightgray,fill=lightgray](30,30) rectangle (35,35);
\draw[red,fill=red](29.5,30) rectangle (30,35);
\draw[green,fill=green](30,35) rectangle (35,35.5);

  \car{(2,31.7)}{black}
  \car{(21,31.7)}{black}
  \car{(26,31.7)}{black}

  \car[270]{(32, 48)}{black}
  \car[270]{(32, 40)}{black}
  \car[270]{(32, 34)}{black}
\end{tikzpicture} \hspace{0.5 cm}
& \begin{tikzpicture}[scale=0.07]
  \draw[black,fill=lightgray](30,00) rectangle (50,80);

  \draw[black,fill=lightgray](0,30) rectangle (80,50);
  \draw[thick,white,dash pattern=on 7 off 4](0,35) to (40,35);
  \draw[thick,white,dash pattern=on 7 off 4](0,40) to (40,40);
  \draw[thick,white,dash pattern=on 7 off 0](0,45) to (40,45);
  \draw[thick,white,dash pattern=on 7 off 4](40,45) to (80,45);
  \draw[thick,white,dash pattern=on 7 off 4](40,40) to (80,40);
  \draw[thick,white,dash pattern=on 7 off 0](40,35) to (80,35);
  \draw[thick,white,dash pattern=on 7 off 4](35,40) to (35,80);
  \draw[thick,white,dash pattern=on 7 off 4](40,40) to (40,80);
  \draw[thick,white,dash pattern=on 7 off 0](45,40) to (45,80);
  \draw[thick,white,dash pattern=on 7 off 0](35,0) to (35,40);
  \draw[thick,white,dash pattern=on 7 off 4](40,0) to (40,40);
  \draw[thick,white,dash pattern=on 7 off 4](45,0) to (45,40);

  \draw[lightgray,fill=lightgray](30,30) rectangle (50,50);
  \draw[red,fill=red](29.5,30) rectangle (30,45);
  \draw[red,fill=red](50.5,50) rectangle (50,35);
  \draw[red,fill=red](50,29.5) rectangle (35,30);
  \draw[green,fill=green](30,50) rectangle (45,50.5);

  \car{(26,31.7)}{black}
  \car{(11,36.7)}{black}
  \car{(18,36.7)}{black}
  \car{(25,36.7)}{black}
  \car{(25,41.7)}{black}
  \car[90]{(43.3,26)}{black}
  \car[90]{(43.3,21)}{black}

  \car[270]{(37, 48)}{black}
  \car[250]{(35, 40)}{black}
  \car[270]{(37, 54)}{black}
  \car[180]{(55, 38.3)}{black}

  \draw[white,thick,->] (37.5,78) to (37.5,72);
  \draw[white,thick,->] (32.5,78) to [out = 270, in = 30] (30.5,72);
  \draw[white,thick,->] (42.5,78) to [out = 270, in = 150] (44.5,72);
  \draw[white,->,thick](2,32.5) to [out = 0, in = 120] (8,30.5);
  \draw[white,->,thick](2,37.5) to [out = 0, in = 180] (8,37.5);
  \draw[white,->,thick](2,42.5) to [out = 0, in = 240] (8,44.5);

  \draw[white,thick,->] (42.5,2) to (42.5,8);
  \draw[white,thick,->] (37.5,2) to [out = 90, in = 300] (35.5,8);
  \draw[white,thick,->] (47.5,2) to [out = 90, in = 210] (49.5,8);
  \draw[white,->,thick](78,37.5) to [out = 180, in = 30] (72,35.5);
  \draw[white,->,thick](78,42.5) to [out = 180, in = 0] (72,42.5);
  \draw[white,->,thick](78,47.5) to [out = 180, in = 300] (72,49.5);
  \end{tikzpicture}
\\
\scriptsize (a) \hspace{0.5 cm} 
&
\scriptsize (b) 
\end{tabular}
\caption{{\small Graphical representations of (a) the two-lane example considered in Subsection~\ref{sec:heavy_two} and (b) the four-lane example in Subsection~\ref{sec:heavy_four}.}}
\label{fig:lanes_setting}
\end{figure}

\subsection{Two-lane example}\label{sec:heavy_two}
First consider an example with two lanes as depicted in Figure~\ref{fig:lanes_setting}(a). Due to the fixed cycle, both lanes operate as independent FCTL queues. The challenge, however, is to determine the optimal capacity allocation that dictates how the cycle time should be divided. In this example, we set the cycle length according to the sum of the green times and we choose an all-red or clearance time of $r_T$ slots.
We consider Poisson arrivals and geometric arrivals, both with a mean arrival rate of $0.4$ vehicles per slot. We further choose $r_T=5$ and consider several values of $c$. 
We determine the optimal $\beta_i$ according to the first-order dimensioning rule in \eqref{eq:sol1} and the refined rule in \eqref{dimrule}. In Table~\ref{t:newNumExampleGs} we display the optimal $\beta_i$ according to the two dimensioning rules together with the resulting green times. 
We see in Table~\ref{t:newNumExampleGs} that the green times only weakly depend on the distribution (Poisson or geometric). This, at least partly, relates to the scaling rule~\eqref{eq:ht_scaling} that we propose: if the mean arrival rate of two vehicle streams is the same (as is the case in this example), the only difference in the green time is caused by differences in the \emph{standard deviation} of the arrival processes and by the parameters $\beta_i$. The latter are the same for all flows under the dimensioning rule when using~\eqref{eq:sol1} and differ only slightly under the dimensioning rule for~\eqref{dimrule}.


\begin{table}[h]
	\centering
	{\small\begin{tabular}{c|cc|cc}
	    &   \multicolumn{2}{c}{Dimensioning rule~\eqref{eq:sol1}  }    & \multicolumn{2}{c}{Dimensioning rule~\eqref{dimrule}} \\ \hline
	$c$ & $g_1$ ($\beta_1$) & $g_2$ ($\beta_{2}$)  & $g_{1}^c$ ($\beta_{1}^c$) & $g_{2}^c$ ($\beta_{2}^c$)\\\hline
	30 & 12.46 (0.132) & 12.54 (0.132) & 12.46 (0.132) & 12.54 (0.132) \\
	50 & 22.29 (0.512) & 22.71 (0.512) & 22.29 (0.511) & 22.71 (0.513) \\
	100 & 46.87 (1.086) & 48.13 (1.086) & 46.84 (1.082) & 48.16 (1.090) \\
	200 & 96.03 (1.792) & 98.97 (1.792) & 95.92 (1.780) & 99.08 (1.803) \\
	500 & 243.5 (3.077) & 251.5 (3.077) & 243.1 (3.049) & 251.9 (3.101) \\\hline
	\end{tabular}
	\caption{{\small Optimal green times and $\beta_i$'s according to Proposition~\ref{thm:noCosts} (rule~\eqref{eq:sol1}) and Theorem~\ref{thm:noCostsHigherOrder}
			(rule~\eqref{dimrule}). For rule \eqref{dimrule} we use the notation  $\beta_i^c$ and $g_i^c$. We consider a Poisson arrival stream with mean $0.4$ at lane $1$, a geometric arrival stream at lane $2$ with mean $0.4$, and $r_T=5$. We study various values of $c$.}}
	\label{t:newNumExampleGs}}
\end{table}

\begin{table}[h]
	\centering
	{\small\begin{tabular}{c|cccc}
		$c$ & $\mathbb{E}[X_{g_1}^{(1)}]$ & Eq.~\eqref{eq:int_mean1} & $\mathbb{E}[X_{g_{1}^c}^{(1)}]$ & Eq.~\eqref{eq:int_meansqrtc} \\\hline
		$30$ & $11.53$ & $11.19$ & $11.53$ & $11.35$ \\
		$50$ & $2.396$ & $2.285$ & $2.402$ & $2.417$ \\
		$100$ & $0.6978$ & $0.6383$ & $0.7066$ & $0.7286$\\
		$200$ & $0.1686$ & $0.1431$ & $0.1742$ & $0.1887$ \\
		$500$ & $0.00609$ & $0.00412$ & $0.00666$ & $0.00960$\\\hline
		$c$ & $\mathbb{E}[X_{g_{2}}^{(2)}]$ & Eq.~\eqref{eq:int_mean1} & $\mathbb{E}[X_{g_{2}^c}^{(2)}]$ & Eq.~\eqref{eq:int_meansqrtc} \\\hline
		$30$ & $13.60$ & $13.24$ & $13.60$ & $13.52$ \\
		$50$ & $2.870$ & $2.704$ & $2.863$ & $2.923$ \\
		$100$ & $0.8577$ & $0.7553$ & $0.8500$ & $0.8930$ \\
		$200$ & $0.2156$ & $0.1693$ & $0.2104$ & $0.2343$ \\
		$500$ & $0.00865$ & $0.00488$ & $0.00801$ & $0.0124$ \\
	\end{tabular}
	\caption{{\small Exact values of the mean overflow queue with the green time based on Proposition~\ref{thm:noCosts}, $\mathbb{E}[X_{g_i}^{(i)}]$, and on Theorem~\ref{thm:noCostsHigherOrder}, $\mathbb{E}[X_{g_{i}^c}^{(i)}]$, respectively. The green times are randomized as in Remark~\ref{rem:nonintegeroptimal}. The table also displays an approximation of the mean overflow queue based on Equation~\eqref{eq:int_mean1} for the $g_i$ and an approximation based on Equation~\eqref{eq:int_meansqrtc} for the $g_{i}^c$. The results are for Poisson arrivals with mean $0.4$ at lane $1$, geometric arrivals with mean $0.4$ at lane $2$, and $r_T=5$. We study various values of $c$.}}
\label{t:newNumExampleEXg}}
\end{table}

The difference between the green times based on the first-order dimensioning rule and the refined dimensioning rule in Table~\ref{t:newNumExampleGs} is generally small. These small differences in the green-time allocations can be explained by the fact that the first-order approximation for the mean overflow queue is already sharp, see Table~\ref{t:newNumExampleEXg}, where we take the various green-time allocations as in Table~\ref{t:newNumExampleGs} while randomizing the green times as in Remark~\ref{rem:nonintegeroptimal}, and compute the exact value and approximations for the mean overflow queue. The minor differences in the green-time allocations in Table~\ref{t:newNumExampleGs} also lead to relatively small differences in the mean overflow queue as can be observed in Table~\ref{t:newNumExampleEXg}. The larger green times are allocated to the flow with the larger standard deviation of the number of arrivals per slot (and thus also to the flow with the larger mean overflow queue), which makes sense intuitively: if there is any excess green time, it should be allocated to the longest queue (within certain boundaries). We also found the optimal integer green-time allocation for the cases studied in Table~\ref{t:newNumExampleGs} and the optimal green times generally agree with the rounded values of the non-integer green times presented in Table~\ref{t:newNumExampleGs}, certainly when using~\eqref{eq:sol2}. Summarizing, both dimensioning rules yield results that are close to the optimum. As a last remark, we note that the first-order rule~\eqref{eq:sol1} is already a good way of dimensioning this two-lane intersection.

\begin{remark}[Comparison with Webster's approach~\cite{webster1958traffic}.]
	A well-known allocation scheme for green times at intersections with fixed settings is derived in~\cite{webster1958traffic}. The approximation scheme is derived using several approximations and yields ``optimal'' cycle lengths and green-time allocations. Our scheme, in contrast, does not provide an optimal cycle length, but rather offers a way to compute close-to-optimal green times \emph{given} a certain cycle length.
	
	When comparing the computation of the green times using Webster's method and ours for the optimal cycle length obtained by Webster, we note that the differences are generally small. In essence, Webster's method allocates green time proportionally to the arrival rate at each lane which corresponds to the $c\mu_i$ term in Equation~\eqref{eq:ht_scaling} (where we append the $i$ to represent the dependence on lane $i$). This term, certainly if $c$ grows, is the dominating factor in our green-time allocation.
	However, larger differences between Webster's method and ours are present when the second term in Equation~\eqref{eq:ht_scaling} is non-negligible compared to the first term. This is, for example, the case if the arrival distribution has a high standard deviation. It is thus important to take the variability of the arrival process into account when looking for optimal green times. 
In Appendix~\ref{app:webster} we provide a more thorough comparison between our method and Webster's.
\end{remark}

\subsection{Four-lane example with weights}\label{sec:heavy_four}

We next consider the influence of weights for an intersection with four approaches, see also Figure~\ref{fig:lanes_setting}(b), again assuming an all-red time of 5 slots. We have four signal phases and focus on the lane with the highest ratio of flow to saturation flow in each signal phase, which is a common approach (cf. \cite{webster1958traffic}).
We apply the dimensioning rule in \eqref{eq:approach2b} and obtain the optimal $\beta_i$ numerically (see Corollary~\ref{cor:weigthedNumerical}).
We take an arrival rate of 0.3 for lanes 1 and 2 and an arrival rate of 0.1 for lanes 3 and 4. The arrival processes vary: geometric arrivals in lane 1, Poisson arrivals in lane 2, and negative binomial arrivals for lanes 3 and 4. For these last two lanes, we increase the variability in the arrivals by selecting the parameters of the arrival distribution such that the variance is equal to 0.4 (compared to 0.1 if it had been a Poisson process).


We show results for equal weights $d_i=1$ in Table~\ref{t:newNumExample2EqualWeight} and unequal weights $d_i=i$ in Table~\ref{t:newNumExample2IncreasingWeight}. We display the green time and the optimal $\beta_i$ for each lane in both tables. With equal weights, the $\beta_i$ are the same and the difference in green times is solely due to differences in the mean and the standard deviation of the arrival process, see Table~\ref{t:newNumExample2EqualWeight}. With unequal weights, the $\beta_i$ increase with the weight $d_i$, as expected, although the influence of the weights on the green times remains limited as can be observed in Table~\ref{t:newNumExample2IncreasingWeight}. This makes sense, since the amount of green time that one can freely allocate is rather limited as well, especially for small $c$. E.g., if $c=30$, we only have one unit of green time to allocate freely (since we need $\mu_T c = 24$ for stabilizing all flows and $r_T$ is $5$). For larger $c$, however, we have more freedom to allocate green time, which leads to (slightly) bigger differences between the optimal green times. This can be verified by comparing the optimal green times for $c=500$ between Tables~\ref{t:newNumExample2EqualWeight} and \ref{t:newNumExample2IncreasingWeight}. When comparing the green times obtained by the two dimensioning rules with the actual optimal green times, computed by brute force, it turns out that these values always coincide with the values from Tables~\ref{t:newNumExample2EqualWeight} and \ref{t:newNumExample2IncreasingWeight}, except for a few cases where there is a single time slot difference due to rounding effects.

\begin{table}[h!]
	\centering
{\small \begin{tabular}{c|cccc}
	$c$ & $g_1$ ($\beta_1$) & $g_2$ ($\beta_2$) & $g_3$ ($\beta_3$) & $g_4$ ($\beta_4$) \\\hline
 30 & 9.256 (0.075) & 9.225 (0.075) & 3.260 (0.075) & 3.260 (0.075) \\
 50 & 16.281 (0.290) & 16.124 (0.290) & 6.298 (0.290) & 6.298 (0.290) \\
 100 & 33.844 (0.615) & 33.371 (0.615) & 13.893 (0.615) & 13.893 (0.615) \\
 200 & 68.969 (1.015) & 67.866 (1.015) & 29.083 (1.015) & 29.083 (1.015) \\
 500 & 174.343 (1.743) & 171.350 (1.743) & 74.653 (1.743) & 74.653 (1.743) \\
	\end{tabular}
\caption{{\small Dimensioning rule~\eqref{eq:approach2b}. Optimal green times and $\beta$'s for the four-lane example, with $r_T=5$ and $d_i=1$ for various values of $c$.}}
	\label{t:newNumExample2EqualWeight}}
	
\end{table}

\begin{table}[h!]
	\centering
	{\small\begin{tabular}{c|cccc}
	$c$ & $g_1$ ($\beta_1$) & $g_2$ ($\beta_2$) & $g_3$ ($\beta_3$) & $g_4$ ($\beta_4$) \\\hline
 30 & 9.166 (0.049) & 9.206 (0.069) & 3.291 (0.084) & 3.336 (0.097) \\
 50 & 15.842 (0.191) & 16.036 (0.268) & 6.454 (0.325) & 6.667 (0.373) \\
 100 & 32.627 (0.421) & 33.154 (0.576) & 14.336 (0.686) & 14.882 (0.772) \\
 200 & 66.538 (0.740) & 67.504 (0.969) & 29.989 (1.117) & 30.970 (1.226) \\
 500 & 169.911 (1.426) & 170.847 (1.702) & 76.332 (1.862) & 77.909 (1.973) \\
\end{tabular}
\caption{{\small Dimensioning rule~\eqref{eq:approach2b}. Optimal green times and $\beta$'s for the four-lane example, with $r_T=5$ and $d_i=i$ for various values of $c$.}}
	\label{t:newNumExample2IncreasingWeight}}
\end{table}


We have also computed the mean overflow queue $\mathbb{E}[X_{g_i}^{(i)}]$ for traffic flows $i=1,2,3,4$ using an exact analysis based on random green times (as described in Remark~\ref{rem:nonintegeroptimal}) and the refined approximation Eq.~\eqref{eq:int_meansqrtc}. The results, shown in Table~\ref{t:newNumExample2EXg}, confirm the high accuracy of the approximation. In Appendix~\ref{app:webster}, we compare our approach with Webster's method for the green-time allocation provided in Table~\ref{t:newNumExample2EqualWeight}.

\begin{table}[h]
	\centering
	{\small\begin{tabular}{c|cc|cc|cc|cc}
\multicolumn{9}{c}{Equal weights}\\
       \hline
		$c$ &  $\mathbb{E}[X_{g_1}^{(1)}]$ & Eq.~\eqref{eq:int_meansqrtc} & $\mathbb{E}[X_{g_2}^{(2)}]$ & Eq.~\eqref{eq:int_meansqrtc} & $\mathbb{E}[X_{g_3}^{(3)}]$ & Eq.~\eqref{eq:int_meansqrtc} & $\mathbb{E}[X_{g_4}^{(4)}]$ & Eq.~\eqref{eq:int_meansqrtc}\\
\hline
		$30$ &     21.422 & 21.158 & 18.805 & 18.492 & 22.192 & 22.304 & 22.192 & 22.304 \\
		$50$ &     5.572 & 5.571 & 4.829 & 4.829 & 6.151 & 6.447 & 6.151 & 6.447 \\
		$100$ &    2.455 & 2.483 & 2.129 & 2.132 & 2.945 & 3.191 & 2.945 & 3.191 \\
		$200$ &    1.181 & 1.206 & 1.011 & 1.025 & 1.559 & 1.737 & 1.559 & 1.737 \\
		$500$ &    0.303 & 0.317 & 0.254 & 0.263 & 0.482 & 0.590 & 0.482 & 0.590 \\
       \hline
\multicolumn{9}{c}{Increasing weights}\\
       \hline
		$c$ &  $\mathbb{E}[X_{g_1}^{(1)}]$ & Eq.~\eqref{eq:int_meansqrtc} & $\mathbb{E}[X_{g_2}^{(2)}]$ & Eq.~\eqref{eq:int_meansqrtc} & $\mathbb{E}[X_{g_3}^{(3)}]$ & Eq.~\eqref{eq:int_meansqrtc} & $\mathbb{E}[X_{g_4}^{(4)}]$ & Eq.~\eqref{eq:int_meansqrtc}\\
\hline
		$30$ &    33.803 & 33.490 & 20.619 & 20.295 & 19.652 & 19.781 & 16.884 & 17.038 \\
		$50$ &     9.456 & 9.451 & 5.349 & 5.374 & 5.342 & 5.636 & 4.476 & 4.786 \\
		$100$ &    4.595 & 4.604 & 2.385 & 2.397 & 2.481 & 2.722 & 2.025 & 2.274 \\
		$200$ &    2.405 & 2.421 & 1.141 & 1.151 & 1.265 & 1.445 & 1.017 & 1.197 \\
		$500$ &    0.660 & 0.675 & 0.281 & 0.291 & 0.379 & 0.484 & 0.302 & 0.404 \\
       \hline
	\end{tabular}
	\caption{{\small The mean overflow queue, exact and approximated with Eq.~\eqref{eq:int_meansqrtc}, for the four-lane example with weights.}}
\label{t:newNumExample2EXg}}
\end{table}

\section{Proof of Theorem~\ref{thm:ht_FCTL} using the transform method} \label{sec:proofFCTLconvergence}

In this section, we present the proof of Theorem~\ref{thm:ht_FCTL}, which we regard as the main mathematical novelty in this paper. The theorem shows weak convergence of the scaled overflow queue to a non-degenerate limit. The general proof structure is explained  in  Subsection~\ref{sec:proofHTFCTLidea} and executed in Subsection~\ref{sec:proofHTFCTL}.

\subsection{Sketch of the proof of Theorem~\ref{thm:ht_FCTL}}\label{sec:proofHTFCTLidea}

Let $X_g(w)$ denote the probability generating function (PGF) of the stationary overflow queue. In~\cite{boon2019pollaczek}, it is shown that there is an $\epsilon_0>0$ such that for all $\epsilon\in(0,\epsilon_0)$
\begin{align}\label{pk1}
&X_g(w)=
\exp\left(\frac{1}{2\pi i}\oint_{|z|=1+\epsilon}\frac{(Y'(z)z-Y(z))(w-Y(w))}{(z-Y(z))(z Y(w) -w Y(z))}
\ln\left(1-\frac{Y^c(z)}{z^g}\right)\mathrm{d}z\right),
\end{align}
for any $|w|<1+\epsilon$ with principal value of the logarithm and where $Y(z)$ is the PGF of the number of arrivals in a single slot. We switch to the moment generating function (MGF) by a change of variables, replacing $w$ by $\exp(t/(\sigma\sqrt{c}))$.

We will prove that the MGF of the FCTL overflow queue converges to the MGF of the $M_\beta$ given by, see~\cite{abate1993calculation},
\begin{equation}
\mathbb{E}(\e^{t M_\beta})=\exp\left(\frac{1}{2\pi i}\int_{\mathcal{C}}\frac{t}{u(t-u)}\ln\left(
1-\e^{-\beta u+\frac12 u^2}\right)\mathrm{d}u\right),
\end{equation}
where $t\in\mathbb{C}$ and $\mathcal{C}$ is a curve going from $-i\cdot\infty$ to $+i\cdot\infty$, passing $t$ to the right. We choose $\mathcal{C} : u=\beta+iv$, $-\infty<v<\infty$, and then we get for $\text{Re}(t)<\beta$ that
\begin{align}\label{eq75}
&\mathbb{E}(\e^{t M_\beta})=\exp\left(\frac{1}{2\pi}\int_{-\infty}^\infty\frac{t}{(\beta+iv)(t-\beta-iv)}\ln\left(
1-\e^{-\frac12\beta^2-\frac12 v^2}\right)\mathrm{d}v\right).
\end{align}
Then, we will prove that
\begin{equation}\label{eq76}
X_g(w)=\mathbb{E}\big(e^{tM_\beta}\big)\left(1+O\big(\frac{1}{\sqrt{c}}\big)\right)
\end{equation}
with $w=\exp(t/(\sigma\sqrt{c}))$, as $c\to\infty$, uniformly in $t$ in any bounded set contained in $\text{Re}(t)\leq\frac12\beta$, proving Equation~\eqref{eq:conv_Xg} in Theorem~\ref{thm:ht_FCTL}. We shall work from the integral
\begin{align}\label{eq77}
&I_c(w):=\frac{1}{2\pi i} \oint_{|z|=1+\epsilon}\frac{Y'(z)z-Y(z)}{z-Y(z)}\frac{w-Y(w)}{z Y(w) -w Y(z)}
\ln\left(1-\frac{Y^c(z)}{z^g}\right)\mathrm{d}z
\end{align}
with $w=\exp(t/(\sigma\sqrt{c}))$, see Equation~\eqref{pk1}, towards the integral
\begin{equation}
J(t):=\frac{1}{2\pi}\int_{-\infty}^\infty\frac{t}{(\beta+iv)(t-\beta-iv)}\ln\left(
1-\e^{-\frac12\beta^2-\frac12 v^2}\right)\mathrm{d}v,
\end{equation}
see Equation~\eqref{eq75}. We do this by using the dedicated saddle point method presented in~\cite{janssen2015novel} for the bulk service queue in heavy traffic. To avoid certain technical complications, we assume, as in~\cite{janssen2015novel}, that the maximum of $|Y(z)|$ over $z$, $|z|=r$, is uniquely achieved at $z=r$ for any $r\in(0,R)$. Under this assumption, see \cite[Sec. 3]{janssen2015novel}, the function
\begin{equation}\label{eq79}
h(z):=-\ln z+\frac{c}{g} \ln Y(z)
\end{equation}
has a unique saddle point $z_{sp}$ in $(1,R)$ with
\begin{equation}
h(z_{sp})<0=h'(z_{sp}), \quad h''(z_{sp})\to\frac{\sigma^2}{\mu}\qquad(c\to\infty),
\end{equation}
and such that $\text{Re}[h(z)]$, $|z|=z_{sp}$, is strictly maximal at $|z|=z_{sp}$. This saddle point converges to 1 as $c\to\infty$, and $z_{sp}<z_0$, where $z_0$ is the zero of $z^g-Y^c(z)$ outside the unit disk of smallest modulus. We shall take $1+\epsilon=z_{sp}$ in Equation~\eqref{eq77}. As $c\to\infty$, we have, due to rapid decay of $|Y^c(z)/z^g|$ along $|z|=z_{sp}$ from $z=z_{sp}$ onwards, that we may restrict the integration over $z$ in Equation~\eqref{eq77} to only a small portion of $|z|=z_{sp}$ near $z=z_{sp}\to1$. Also, we have $w=\exp(t/(\sigma\sqrt{c}))\to1$, $c\to\infty$, since $t$ is in a bounded set.

Our proof has the following main steps.

\begin{enumerate}
	\item[I.] Approximating the integral in Equation~\eqref{eq77}
	\begin{equation}\label{eq711}
	\frac{Y'(z)z-Y(z)}{z-Y(z)}\frac{w-Y(w)}{z Y(w) -w Y(z)} \text{ by }
	\frac{w-1}{(z-1)(w-z)}
	\end{equation}
	for $z$ and $w$ near 1.
	
	\item[II.] \label{step2} Substituting $z=z(x)$, $-\delta\leq x\leq\delta$ with $z(0)=z_{sp}$ to achieve that
	\begin{equation}\label{eq712}
	\frac{Y^c(z(x))}{(z(x))^g}=\exp\left(gh(z_{sp})-\frac12gh''(z_{sp})x^2\right)
	\end{equation}
	assumes the form of a Gaussian (steepest descent curve).
	
	\item[III.] Showing that
	\begin{equation}\label{eq713}
	gh(z_{sp}) \to -\frac12\beta^2,\, h''(z_{sp})\to\frac{\sigma^2}{\mu}
	\end{equation}
	as $c\to\infty$. Substituting $v=x\sqrt{gh''(z_{sp})}$, $-\delta\leq x\leq \delta$, we see from Equations~\eqref{eq712} and \eqref{eq713}, that we approximate
	\begin{equation}\label{eq:log}
	\ln\left(1-\frac{Y^c(z(x))}{(z(x))^g}\right) \text{ by } \ln\left(1- \e^{-\frac12\beta^2-\frac12 v^2}\right)
	\end{equation}
	as $c\to\infty$.
	
	\item[IV.] Showing that the combined effect on $(w-1)/\left((z-1)(w-z)\right)$ in Equation~\eqref{eq711} of the substitutions $z=z(x)$, $v=x\sqrt{gh''(z_{sp})}$ amounts to approximating
	\begin{equation}
	\frac{(w-1) \mathrm{d} x}{(z-1)(w-z)} \text{ by } \frac{t \mathrm{d} v}{(\beta+iv)(t-\beta-iv)},
	\end{equation}
	where $w=\exp(t/(\sigma\sqrt{c}))$ and $c\to\infty$.
	
	\item[V.] The completion of the proof of Equation~\eqref{eq76}.
\end{enumerate}


\subsection{Full proof of  Theorem~\ref{thm:ht_FCTL}}\label{sec:proofHTFCTL}

We shall next present the details for the five main steps.

\noindent\textbf{Step I.} We have in $|z-1|\leq \frac12(R-1)=:\delta$
\begin{equation}\label{eq717}
Y(z)=1+\mu(z-1)+O(|z-1|^2), Y'(z)=\mu+O(|z-1|),
\end{equation}
so that
\begin{align}
&z-Y(z)=(1-\mu)(z-1)(1+O(|z-1|),\\& z Y'(z)-Y(z)=-(1-\mu)(1+O(|z-1|)).
\end{align}
Therefore, in a set of $z$'s, $|z-1|\leq\delta_1$ with $\delta_1>0$,
\begin{equation}\label{eq719}
\frac{zY'(z)-Y(z)}{z-Y(z)}=\frac{-1}{z-1}\big(1+O(|z-1|)\big).
\end{equation}
We shall show below that for $|z-1|$, $|w-1|\leq \frac{1}{2}(R-1)=\delta$
\begin{align}
zY(w)-wY(z)&=(1-\mu)(z-w)(1+O(|z-1|+|w-1|)).\label{eq:23a}
\end{align}
Therefore, also using the first item in Equation~\eqref{eq717} with $w$ instead of $z$,
\begin{equation}\label{eq721}
\frac{w-Y(w)}{zY(w)-wY(z)}=\frac{(1-\mu)(w-1)(1+O(|w-1|))}{(1-\mu)(z-w)(1+O(|z-1|+|w-1|))}
\end{equation}
holds in a set of $z$'s and $w$'s, $|z-1|\leq\delta_2$ and $|w-1|\leq\delta_2$ with $\delta_2>0$. Combining Equations~\eqref{eq719} and \eqref{eq721}, we get
\begin{align}\label{eq722}
&\frac{Y'(z)z-Y(z)}{z-Y(z)}\frac{w-Y(w)}{z Y(w) -w Y(z)}=\frac{w-1}{(z-1)(w-z)}\big(1+O(|z-1|+|w-1|)\big),
\end{align}
holding in a set of $z$'s and $w$'s, $|z-1|\leq\delta_3$ and $|w-1|\leq\delta_3$ with $\delta_3>0$.

We finally show that Equation~\eqref{eq:23a} holds when $|z-1|$, $|w-1|\leq\delta$. We have
\begin{equation}
	Y(v) = 1+\mu(v-1)+\sum_{k=2}^\infty c_k(v-1)^k,
\end{equation}
for $|v-1|<\delta$ and where $\sum_{k=2}^\infty |c_k(v-1)^k|\leq \sum_{k=2}^\infty k|c_k|\delta^k<\infty$, and so
\begin{equation}\label{eq:26b}
	zY(w) - wY(z) = (1-\mu)(z-w)+\sum_{k=2}^\infty c_k\left(z(w-1)^k-w(z-1)^k\right).
\end{equation}
For $k=2,3,...$ we have
\begin{align}
	z&(w-1)^k-w(z-1)^k = (w-1)^k-(z-1)^k+(z-1)(w-1)\left((w-1)^{k-1}-(z-1)^{k-1}\right).
\end{align}
Using $a^n-b^n = (a-b)\sum_{i=0}^{n-1}a^i b^{n-1-i}$ with $a=w-1$, $b=z-1$ and $n=k,k-1$ we get
\begin{align}
\label{eq:26d}
	z&(w-1)^k-w(z-1)^k = (w-z)\left[\sum_{j=0}^{k-1}(w-1)^j(z-1)^{k-1-j}+\sum_{j=0}^{k-2}(w-1)^{j+1}z^{k-1-j}\right].
\end{align}
Let $m=\max\{|z-1|,|w-1|\}$. The modulus of the quantity within the $[]$ at the right-hand side of Equation~\eqref{eq:26d} is bounded by
\begin{equation}
	km^{k-1}+(k-1)m^k \leq \left(|z-1|+|w-1|\right)(k\delta^{k-2}+(k-1)\delta^{k-1})
\end{equation}
since $m\leq |z-1|+|w-1|$ and $|z-1|$, $|w-1|\leq \delta$. Therefore
\begin{align}\label{eq:26f}
	\Bigg|\sum_{k=2}^\infty& c_k\left(z(w-1)^k-w(z-1)^k\right)\Bigg|\leq\\& |z-w|\left(|z-1|+|w-1|\right)\sum_{k=2}^\infty |c_k|\left(k\delta^{k-2}+(k-1)\delta^{k-1}\right).
\end{align}
The infinite series at the right-hand side of Equation~\eqref{eq:26f} has a finite value and does not depend on $z,w$ when $|z-1|$, $|w-1|\leq\delta$. From this and Equation~\eqref{eq:26b} we get Equation~\eqref{eq:23a} for such $z,w$.

\noindent\textbf{Step II.} We have
\begin{equation}
\frac{Y^c(z)}{z^g}=\exp\big(gh(z)\big),
\end{equation}
with $h(z)$ given by Equation~\eqref{eq79}. We define $z=z(x)$ for real $x$ of small modulus by setting
\begin{equation}
h(z(x))=h(z_{sp})-\frac12x^2h''(z_{sp}).
\end{equation}
In \cite[Section 3]{janssen2015novel}, it is shown that there is a $\delta>0$, independent of $c\geq1$, such that $z(x)$ is given by a power series
\begin{equation}
z(x)=z_{sp}+ix+\sum_{k=2}^\infty c_k(ix)^k, \qquad |x|\leq \delta,
\end{equation}
with real $c_k$. Thus $z'(x)=i+O(|x|)$, which shows that the curve $(x, z(x))$ is tangent to the circle $|z|=z_{sp}$ at $z=z_{sp}$.

Substituting $z=z(x)$, $-\delta\leq x\leq\delta$, in Equation~\eqref{eq77} produces an approximation of $I_c(w)$ with exponentially small error. Note that $\mathrm{d} z=z'(x)\mathrm{d} x=\big(i+O(|x|)\big)\mathrm{d} x$. When we use, furthermore, Equation~\eqref{eq722}, we get
\begin{equation}\label{eq726}
I_c(w)=\frac{1}{2\pi} \int_{-\delta}^\delta \frac{w-1}{(z(x)-1)(w-z(x))}
\ln\left(1-\frac{Y^c(z(x))}{(z(x))^g}\right)(1+O)\mathrm{d}x,
\end{equation}
where $O$ abbreviates $O\big(|x|+|z(x)-1|+|w-1|\big)$. Note that $\frac{Y^c(z(x))}{(z(x))^g}$ is given by Equation~\eqref{eq712} in Gaussian form.

\noindent\textbf{Step III.} We have that $z_{sp}$ is the solution of $h'(z)=0$ with $z$ larger than, but close to, 1. From
\begin{equation}
0=h'(z_{sp})=a_1+a_2(z_{sp}-1)+\frac12a_3(z_{sp}-1)^2+\dots,
\end{equation}
where $a_i=h^{(i)}(1)$, we get
\begin{align}\label{eq728}
z_{sp}-1 &=\frac{-a_1/a_2}{1+a_3(z_{sp}-1)/2a_2+\ldots}-\frac{a_1}{a_2}+\frac{a_1a_3}{2a_2^2}(z_{sp}-1)+\dots \\
&= -\frac{a_1}{a_2}-\frac{a_3}{2a_2}\left(\frac{a_1}{a_2}\right)^2+\dots.
\end{align}
Next, from Equation~\eqref{eq728}, using $h(1)=0$, we get
\begin{align}
h(z_{sp}) &=a_1(z_{sp}-1)+\frac12a_2(z_{sp}-1)^2+\frac16a_3(z_{sp}-1)^3+\dots\\
&=-\frac{a_1^2}{2a_2}-\frac{a_3a_1^3}{6a_2^3}-\dots.
\end{align}
We express $a_i=h^{(i)}(1)$, $i=1,2,3$, in terms of $\mu$, $\sigma$, $\beta$ and $c$. We have
\begin{equation}
h'(z)=-\frac{1}{z}+\frac{c}{g}\frac{Y'(z)}{Y(z)},
\end{equation}
and so, from $g=c\mu+\beta\sigma\sqrt{c}$ and $Y(1)=1$, $Y'(1)=\mu$,
\begin{equation}
a_1=h'(1)=\frac{c\mu}{g}-1=-\beta\sigma\frac{\sqrt{c}}{g}=\frac{-\beta\sigma}{\mu\sqrt{c}} \left(1+O\Big(\frac{1}{\sqrt{c}}\Big)\right).
\end{equation}
Next,
\begin{equation}
h''(z)=\frac{1}{z^2}+\frac{c}{g}\frac{Y''(z)Y(z)-(Y'(z))^2}{(Y(z))^2},
\end{equation}
and so
\begin{align}
a_2=h''(1)&=1+\frac{1}{\mu}\left(1+O\Big(\frac{1}{\sqrt{c}}\Big)\right)\left(Y''(1)-\big(Y'(1)\big)^2\right)\\
&=\frac{1}{\mu}\left(Y''(1)+\mu-\mu^2\right)+O\Big(\frac{1}{\sqrt{c}}\Big)
=\frac{\sigma^2}{\mu}+O\Big(\frac{1}{\sqrt{c}}\Big).
\end{align}
In a similar fashion, $a_3=h'''(1)$ can be computed as a quantity that remains bounded as $c\to\infty$.

We then find, subsequently,
\begin{align}\label{eq734}
z_{sp}-1&=\frac{\beta}{\sigma\sqrt{c}}\left(1+O\Big(\frac{1}{\sqrt{c}}\Big)\right),\, h(z_{sp})=\frac{-\beta^2}{2c\mu}\left(1+O\Big(\frac{1}{\sqrt{c}}\Big)\right), \\
h''(z_{sp})&=h''(1)+O(z_{sp}-1)=\frac{\sigma^2}{\mu}\left(1+O\Big(\frac{1}{\sqrt{c}}\Big)\right).
\end{align}
It then follows that
\begin{equation}\label{eq736}
gh(z_{sp})=-\frac12\beta^2+O\Big(\frac{1}{\sqrt{c}}\Big),\, h''(z_{sp})=\frac{\sigma^2}{\mu}\left(1+O\Big(\frac{1}{\sqrt{c}}\Big)\right).
\end{equation}
For later use in Step IV, we also mention that
\begin{equation}\label{eq737}
\frac{\sqrt{gh''(z_{sp})}}{\sigma\sqrt{c}}=1+O\Big(\frac{1}{\sqrt{c}}\Big),\, (z_{sp}-1)\sqrt{gh''(z_{sp})}=\beta\left(1+O\Big(\frac{1}{\sqrt{c}}\Big)\right).
\end{equation}
Note that for $-\delta\leq x\leq\delta$ we have from Equation~\eqref{eq736}
\begin{align}\label{eq738}
\ln\left(1-\frac{Y^c(z(x))}{(z(x))^g}\right) &=\ln\left(1-\exp\big(gh(z_{sp})-\frac12gh''(z_{sp})x^2\big)\right)\\
&=\ln\left(
1-\e^{-\frac12\beta^2-\frac12 v^2}\right)\left(1+O\Big(\frac{1}{\sqrt{c}}\Big)\right),
\end{align}
where we have set $v=x\sqrt{gh''(z_{sp})}$.

\noindent\textbf{Step IV.} Let $t$ be in a bounded set with $\text{Re}(t)\leq \frac12\beta$. Then
\begin{equation}
w-1=\exp(t/(\sigma\sqrt{c}))-1=\frac{t}{\sigma\sqrt{c}}\left(1+O\Big(\frac{1}{\sqrt{c}}\Big)\right).
\end{equation}
With $z=z(x)=z_{sp}+ix+O(x^2)$, we have
\begin{align}\label{eq740}
&\frac{w-1}{(z-1)(w-z)}=\frac{t/(\sigma\sqrt{c})}{\big(z_{sp}-1+ix\big) \big(\frac{t}{\sigma\sqrt{c}}-(z_{sp}-1)-ix\big)}\left(1+O\Big(|x|+\frac{1}{\sqrt{c}}\Big)\right).
\end{align}
The relative error factor $1+O\Big(|x|+\frac{1}{\sqrt{c}}\Big)$ follows from Equation~\eqref{eq734} and $\text{Re}(t)\leq\frac12\beta$, so that
\begin{equation}
z_{sp}-1-\text{Re}\left(\frac{t}{\sigma\sqrt{c}}\right)\geq \frac{\beta}{2\sigma\sqrt{c}}\left(1+O\Big(\frac{1}{\sqrt{c}}\Big)\right).
\end{equation}
We next substitute $v=x\sqrt{gh''(z_{sp})}$. Writing
\begin{equation}
\gamma=\frac{\sqrt{gh''(z_{sp})}}{\sigma\sqrt{c}}, \,\eta=(z_{sp}-1)\sqrt{gh''(z_{sp})},
\end{equation}
we have uniformly in $x\in\mathbb{R}$
\begin{align}\label{eq743}
&\frac{t/(\sigma\sqrt{c})\,\mathrm{d}x}{\big(z_{sp}-1+ix\big) \big(\frac{t}{\sigma\sqrt{c}}-(z_{sp}-1)-ix\big)}\\&=\frac{\gamma t \mathrm{d}v}{(\eta+iv)(\gamma t-\eta-iv)}= \frac{t\mathrm{d}v}{(\beta+iv)(t-\beta-iv)}\left(1+O\Big(\frac{1}{\sqrt{c}}\Big)\right),
\end{align}
where we have used Equation~\eqref{eq737}.

\noindent\textbf{Step V.} We have by Equations~\eqref{eq726}, \eqref{eq738}, \eqref{eq740} and \eqref{eq743}
\begin{align}
I_c(w) =\frac{1}{2\pi}&\int_{-\Delta}^\Delta\left[\frac{t}{(\beta+iv)(t-\beta-iv)}\ln\left(
1-\e^{-\frac12\beta^2-\frac12 v^2}\right)\left(1+O\Big(\frac{1+|v|}{\sqrt{c}}\Big)\right)\right] \mathrm{d}v,
\end{align}
where $\Delta=\delta\sqrt{gh''(z_{sp})}$. For this it has been used that
\begin{align}
&|x|=\frac{|v|}{\sqrt{gh''(z_{sp})}}=O\left(\frac{|v|}{\sqrt{c}}\right), \\&
|z(x)-1|\leq |z_{sp}-1|+O(|x|)=O\left(\frac{1+|v|}{\sqrt{c}}\right).
\end{align}
Finally, since
\begin{equation}
\ln\left(1-\e^{-\frac12\beta^2-\frac12 v^2}\right)= O\left(e^{-\frac12v^2}\right)
\end{equation}
while $\Delta=\delta\sigma\sqrt{c}\left(1+O\Big(\frac{1}{\sqrt{c}}\Big)\right) \to\infty$ like $\sqrt{c}$, we get that
\begin{align}
&I_c(w) =\frac{1}{2\pi}\int_{-\infty}^\infty\frac{t}{(\beta+iv)(t-\beta-iv)}\ln\left(
1-\e^{-\frac12\beta^2-\frac12 v^2}\right)\mathrm{d}v \left(1+O\Big(\frac{1}{\sqrt{c}}\Big)\right).
\end{align}
That is, $I_c(w)=J(t)\left(1+O\Big(\frac{1}{\sqrt{c}}\Big)\right)$, and this holds uniformly in $t$ in any bounded set with $\text{Re}(t)\leq \frac12\beta$, finishing the proof of Equation~\eqref{eq76}.\\

Turning to Equation~\eqref{eq:conv_moments} in Theorem~\ref{thm:ht_FCTL}, we have for the MGF's $F_c$ and $F$ in Equation~\eqref{eq76}
\begin{equation}
F_c(t)=\sum_{k=0}^\infty \frac{m_k(c)}{k!}t^k, \quad F(t)=\sum_{k=0}^\infty \frac{m_k}{k!}t^k,
\end{equation}
where $m_k(c)$ and $m_k$ are the $k^\text{th}$ moment of $X_g/(\sigma\sqrt{c})$ and $M_\beta$, respectively. By Cauchy's integral formula for the $k^\text{th}$ derivative at 0 of an analytic function, we have
\begin{equation}
\frac{m_k(c)}{k!}=\frac{1}{2\pi i}\oint_{|t|=a}\frac{F_c(t)}{t^{k+1}}\mathrm{d}t,
\end{equation}
where we take $a>0$ such that the disk $|t|\leq a$ is contained in the set of $t$'s where the convergence in Equation~\eqref{eq76} is uniform. 
Since $F_c(t)=F(t)\left(1+O\Big(\frac{1}{\sqrt{c}}\Big)\right)$ uniformly on $|t|=a$, this yields $m_k(c)=m_k\left(1+O\Big(\frac{1}{\sqrt{c}}\Big)\right)$, proving Equation~\eqref{eq:conv_moments} in Theorem~\ref{thm:ht_FCTL}.\\

To prove Equation~\eqref{eq:conv_prob} in Theorem~\ref{thm:ht_FCTL}, we must argue differently. Letting $t\to-\infty$ in Equation~\eqref{eq75}, we have
\begin{equation}
\mathbb{P}(M_\beta=0)=\exp\left(\frac{1}{2\pi}\int_{-\infty}^\infty \frac{1}{\beta+iv}\ln\Big(1-e^{-\frac12\beta^2-\frac12v^2}\Big)\mathrm{d}v\right).
\end{equation}
Also, setting $w=0$ in Equation~\eqref{pk1}, we have that the front factor in the integral in Equation~\eqref{pk1} is given by
\begin{equation}
\frac{Y'(z)z-Y(z)}{z-Y(z)}\cdot\frac{-1}{z}=\frac{1}{z-1}\Big(1+O\big(|z-1|\big)\Big),
\end{equation}
where $Y(0)>0$ and Equation~\eqref{eq719} have been used. We are now in a completely similar, and indeed even simpler, situation as before:
\begin{align}
&\mathbb{P}\left(\frac{1}{\sigma\sqrt{c}}X_g=0\right)=\exp\left(\frac{1}{2\pi i}\oint_{|z|=1+\epsilon}\frac{1+O(|z-1|)}{z-1}\ln\Big(1-\frac{Y^c(z)}{z^g}\Big)\mathrm{d}z\right).
\end{align}
The combined effect on the front factor of the two substitutions $z=z(x)$ and $v=x\sqrt{gh''(z_{sp})}$ amounts to
\begin{equation}
\frac{1}{i}\frac{\mathrm{d} z}{z-1}=\frac{\mathrm{d} v}{\beta+iv}\left(1+O\Big(\frac{1}{\sqrt{c}}\Big)\right),
\end{equation}
and this yields $\mathbb{P}(X_g/(\sigma\sqrt{c})=0)=\mathbb{P}(M_\beta=0)\left(1+O\Big(\frac{1}{\sqrt{c}}\Big)\right)$.

\begin{remark}\label{rem:noninteger}
	The randomness in $G_i$ as introduced in Remark~\ref{rem:nonintegeroptimal} has a minor impact on the proof of Theorem~\ref{thm:ht_FCTL}. We need to modify Equation~\eqref{eq712} slightly and multiply the left-hand side of Equation~\eqref{eq712} with $1/(p +(1-p)z(x))$ with $p$ as in Remark~\ref{rem:nonintegeroptimal}. Observe that
	\begin{equation} \label{eq:randomGreen}
		1/(p +(1-p)z(x))=1+O\left(z(x)-1\right)
	\end{equation}
	uniformly in $p$ for $0\leq p \leq 1$. As the right-hand side of Equation~\eqref{eq712} is smaller than $1$, see Equation~\eqref{eq713}, we may take the factor in Equation~\eqref{eq:randomGreen} out of the logarithm in \eqref{eq:log}. In this way, the proof of Theorem~\ref{thm:ht_FCTL} still works with the only further modification  that Equation~\eqref{eq726} gets an additional $O\left(z(x)-1\right)$ term from Equation~\eqref{eq:randomGreen}.
\end{remark}

\section{Conclusion}\label{sec:con}

The main technical novelty in this paper concerns establishing heavy-traffic limits for the single-lane FCTL queue, in particular Theorem~\ref{thm:ht_FCTL}. These heavy-traffic limits follow from combining a suitable large-cycle regime \eqref{eq:ht_scaling} with the transform method for establishing convergence in distribution of the stationary overflow to a nondegenerate limit. We are able to use this transform method thanks to the recently obtained  contour-integral expression in~\cite{boon2019pollaczek} for the probability generating function of the overflow. The proof that exploits this transform method is presented in Section~\ref{sec:proofFCTLconvergence} and interesting in its own right. The key technical novelty, the asymptotic expansion of the complex contour integral, is likely to be of broader interest and not limited to the FCTL queue. Examples where this proof method applies include the bulk service queue~\cite{boon2019pollaczek}, extensions of the FCTL queue considered in~\cite{oblakova2019exact}, and stochastic economic lot-scheduling problems~\cite{winands2011stochastic}.

The limiting heavy-traffic behavior is governed by a reflected Gaussian random walk with negative drift, a stochastic process for which many results are available in the literature. This gives heavy-traffic approximations that reduce the complexity of the (pre-limit) expressions for the mean overflow queue in the FCTL queue considerably. These limiting results enable us to formulate easy-to-calculate approximations and allow us to solve capacity allocation problems in the form of optimization problems that generate (close-to-optimal) green times. This adds to the literature of capacity allocation problems
\cite{kleinrock1976queueing,wein1989capacity} and asymptotic dimensioning of queueing systems~\cite{borst2004dimensioning,van2019economies}.

In some practical situations, it might be beneficial to have non-static signaling strategies, such as vehicle-actuated strategies. Some generalizations of the results to non-deterministic cycle times and green times are possible. Under appropriate adaptations of Equation~\eqref{eq:ht_scaling} and certain restrictions on the variance, similar heavy-traffic results can be established as the ones derived in this paper. Another example is vehicle-actuated signaling, where the green times depend on the queue lengths. An example would be that, instead of a fixed green time, we introduce a maximum green time and switch to the next queue as soon as either the queue empties or the maximum green time is reached. The corresponding model is higher-dimensional (as opposed to the one-dimensional FCTL queue) and a theoretical analysis similar to the one conducted here is therefore not possible. Nevertheless, simulations in~\cite{timmerman2020new} show that the same scaling as in rule~\eqref{eq:ht_scaling} leads to similar asymptotic results for vehicle-actuated settings.\\


\paragraph{Acknowledgements}

The work in this paper is supported by the Netherlands Organization for Scientific Research (NWO) under grant number 438-13-206. The work of JvL is further supported by an NWO Vici grant. The funding resource had no involvement in the design and execution of the study.


%
\newpage
\appendix
\section{Remaining proofs}\label{sec:remaining}
 We now provide the proofs of Proposition~\ref{thm:mean} and Theorem~\ref{thm:noCostsHigherOrder}.

\subsection*{Proof of the heavy-traffic approximation for the mean queue length}\label{sec:proofApprox}
\begin{proof} We start with an outline of the proof of Proposition~\ref{thm:mean}. The expression for $\mathbb{E}[X_{g}]$ is close to the expression for the mean overflow queue in the BSQ. In \cite[Sections 4 and 5]{janssen2015novel}, a similar type of result is obtained for the BSQ and we will provide a proof that works along the same lines. In fact, we rewrite and approximate the expression $\mathbb{E}[X_{g}]$ in such a way that we are able to reuse the proof (and subsequent results) of Theorem 3 from Equation (4.3) onward in Section 4 of~\cite{janssen2015novel}. As a last step, we simplify the expressions obtained in~\cite{janssen2015novel} slightly.

We start the proof with an expression for $\mathbb{E}[X_{g}]$. Equation (3.1) of~\cite{boon2019pollaczek} reads
\begin{equation}
\mathbb{E}[X_g] = \frac{1}{2\pi i}\oint_{|z|=1+\epsilon}\frac{Y(z)-zY^\prime(1)}{Y(z)-z}\frac{\left(z^g-Y(z)^c\right)^\prime}{z^g-Y(z)^c}\mathrm{d}z,
\end{equation}
for some $\epsilon > 0$. We define, as before,
\begin{equation}
h(z) = -\ln z +\frac{c}{g}\ln Y(z) .
\end{equation}
Then we are able to derive (following the same steps as in the proof of Lemma 1 in~\cite{janssen2015novel})
\begin{align}
\mathbb{E}[X_g] & = \frac{1}{2\pi i}\oint_{|z|=1+\epsilon}\frac{Y(z)-zY^\prime(1)}{Y(z)-z}\frac{g z^{g-1}-c Y(z)^{c-1}Y^\prime(z)}{z^g-Y(z)^c}\mathrm{d}z \\
& = \frac{1}{2\pi i}\oint_{|z|=1+\epsilon}\left[\frac{Y(z)-zY^\prime(1)}{Y(z)-z} \left(\frac{g}{z}-\frac{g}{z}\left(\frac{c}{g}\frac{zY^\prime(z)}{Y(z)}-1\right)\frac{z^{-g}Y(z)^c}{1-z^{-g}Y(z)^c}\right)\right]\mathrm{d}z\\
& = \frac{g}{2\pi i} \oint_{|z|=1+\epsilon} h^\prime(z)\frac{Y(z)-zY^\prime(1)}{z-Y(z)} \frac{\exp(g h(z))}{1-\exp(gh(z))}\mathrm{d}z,
\end{align}
where in the last step we use that
\begin{align}
& h^\prime(z) = \frac{c}{g}\frac{Y^\prime(z)}{Y(z)}-\frac{1}{z},\\ &\oint_{|z|=1+\epsilon}\frac{Y(z)-zY^\prime(1)}{Y(z)-z}\frac{g}{z}\mathrm{d}z = 0.
\end{align}

We let $z_{sp}$ denote the unique minimum of $h(z)$ with $z\geq 1$ and we let
\begin{equation}
z(x) = z_{sp} + i x +\sum_{k=2}^\infty c_k (ix)^k
\end{equation}
solve the equation
\begin{equation}
h(z(x)) = h(z_{sp})-\frac{1}{2}x^2h^{\prime\prime}(z_{sp})=: q(x).
\end{equation}
Then, following the same steps as are taken in Section 3 of~\cite{janssen2015novel}, we get that with exponentially small error
\begin{equation}\label{eq:int}
\mathbb{E}[X_g] = -\frac{gh^{\prime\prime}(z_{sp})  }{2\pi i} \int_{-1/2\delta}^{1/2\delta} x\frac{Y(z(x))-z(x)Y^\prime(1)}{z(x)-Y(z(x))} \frac{\exp(gq(x))}{1-\exp(gq(x))} \mathrm{d}x
\end{equation}
for some $\delta>0$.

Proceeding as in the proof of Theorem 3 of~\cite{janssen2015novel}, we obtain, since $z(-x) = \overline{z(x)}$ for real $x$, where $\overline{a}$ denotes the complex conjugate of $a$, that
\begin{align}
&x\frac{Y(z(x))-z(x)Y^\prime(1)}{z(x)-Y(z(x))} - x \frac{Y(z(-x))-z(-x)Y^\prime(1)}{z(-x)-Y(z(-x))} =\frac{-2 ix^2\left(1+O(z_{sp}-1)+x^2\right)}{(z_{sp}-1)^2+x^2-2c_2(z_{sp}-1)x^2},
\end{align}
for $|x|\leq 1/2\delta$ and where $c_2\in\mathbb{R}$. This implies that, using the previous result together with Equation~\eqref{eq:int}  and extending the integration range from to $(-\infty, \infty)$ while using symmetry of $q(x)$, we get, with exponentially small error, that
\begin{align}
\mathbb{E}[X_g] = & \frac{gh^{\prime\prime}(z_{sp})}{\pi}\int_{0}^\infty \left[\frac{x^2\left(1+O(z_{sp}-1)+x^2\right)}{(z_{sp}-1)^2+x^2-2c_2(z_{sp}-1)x^2}\frac{\exp(gq(x))}{1-\exp(gq(x))}\right] \mathrm{d}x,
\end{align}
so we are now exactly in the same situation as that of Sections 4 and 5.1 of~\cite{janssen2015novel}. Here it should be noted that the FCTL relation $g=c\mu + \beta\sigma\sqrt{c}$, see Equation~\eqref{eq:ht_scaling}, can be written in the BSQ form of~\cite{janssen2015novel}, $c/g=(1-\gamma/\sqrt{g})/\mu$ with
\begin{equation}\label{eq:gamma}
\gamma = \frac{\beta\sigma}{\sqrt{\mu}}\left(1+\frac{\beta\sigma}{\mu\sqrt{c}}\right)^{-1/2}.
\end{equation}
Hence, letting
\begin{equation}\label{eq:b02}
b_{0}^2 = b(\beta)^2=\frac{\gamma^2\mu}{2\sigma^2}=\frac{1}{2}\beta^2\left(1+\frac{\beta\sigma}{\mu\sqrt{c}}\right)^{-1},
\end{equation}
see Equation~(4.12) of~\cite{janssen2015novel}, we get with an absolute error of order $1/\sqrt{c}$,
\begin{align}
\mathbb{E}[X_g] =& \frac{\sigma}{\pi}\sqrt{\frac{2g}{\mu}} \label{eq:mean_long} G_0(b_0)+\frac{\sigma}{\pi}\sqrt{\frac{2}{\mu}}\left((C_1+C_3)G_0(b_0)-(C_2+b_{0}^2C_3)G_3(b_0)+C_4G_4(b_0)\right),
\end{align}
according to Equation (5.14) of~\cite{janssen2015novel}, with
\begin{align}
	G_3(b) & = \int_{0}^\infty \frac{t^2}{(b^2+t^2)^2}\frac{\e^{-b^2-t^2}}{1-\e^{-b^2-t^2}}\mathrm{d}t\\
	G_4(b) & =\frac{t^2}{b^2+t^2}\frac{\e^{-b^2-t^2}}{(1-\e^{-b^2-t^2})^2}\mathrm{d}t.
\end{align}



We proceed by computing the $C_i$ explicitly. From~\cite{janssen2015novel}, Equations~(5.2-5.5), (5.8), and (5.9), we get
\begin{align}
    C_1&=-\frac{\gamma(\sigma^2-\mu)}{2\sigma^2},\\
    C_2&=\frac{\gamma(\sigma^2-\mu)}{\sigma^2}b_{0}^2.
\end{align}
Furthermore, from~\cite{janssen2015novel}, Equations (5.2-3), (5.5-6), and (5.10), we get
\begin{equation}\label{eq:C3}
    C_3 = -\frac{1}{3} \gamma a \frac{\mu^2}{\sigma^4},
\end{equation}
while from~\cite{janssen2015novel}, Equations (5.2-3), (5.7), and (5.11), we get
\begin{equation}\label{eq:C4}
    C_4=-\gamma \frac{\sigma^2-\mu}{\sigma^2}b_{0}^2+\frac{1}{3}\gamma a \frac{\mu^2}{\sigma^4}b_{0}^2.
\end{equation}
In Equations~\eqref{eq:C3} and \eqref{eq:C4}, $a$ is given by, see~\cite{janssen2015novel}, Equation (5.3),
\begin{align}
    a=a_3(s=\infty) & = - 2+\frac{Y^{\prime\prime\prime}(1)}{Y^\prime(1)}-3Y^{\prime\prime}(1)+2\left(Y^\prime(1)\right)^2\\
    & = \frac{1}{\mu}\left(\mu_3-\mu^3-3(1+\mu)\sigma^2\right),
\end{align}
where $\mu_3=\mathbb{E}[Y^3]$, as in Equation~\eqref{eq:a}.

It follows that
\begin{align}
    & C_1+C_3 = \frac{1}{2b_{0}^2}C_4,\\
    & C_2+b_{0}^2C_3 = -C_4.
\end{align}
When we also use Equations (5.17) and (4.27) from~\cite{janssen2015novel}, we get
\begin{align}
    G_3(b_0)+G_4(b_0) & = \frac{1}{2b_{0}^2} G_2(b_0),\\
    G_0(b_0)+G_2(b_0) & = G_1(b_0),
\end{align}
with
\begin{align*}
	G_2(b) = \int_{0}^\infty \frac{b^2}{b^2+t^2}\frac{\e^{-b^2-t^2}}{1-\e^{-b^2-t^2}}\mathrm{d}t,
\end{align*}
and find
\begin{equation}
    (C_1+C_3)G_0(b_0)-(C_2+b_{0}^2C_3)G_3(b_0)+C_4G_4(b_0) = (C_1+C_3)G_1(b_0).
\end{equation}
Therefore we get, with an absolute error of order $1/\sqrt{c}$,
\begin{equation}
    \mathbb{E}[X_g] = \frac{\sigma}{\pi} \sqrt{\frac{2g}{\mu}}G_0(b_0)+\frac{\sigma}{\pi} \sqrt{\frac{2}{\mu}}(C_1+C_3)G_1(b_0).
\end{equation}

Finally, we have from $g=c\mu+\beta\sigma \sqrt{c}$ and Equations~\eqref{eq:C3} and \eqref{eq:C4} that

\begin{equation}
    \frac{\sigma}{\pi} \sqrt{\frac{2g}{\mu}} = \frac{\sqrt{2}}{\pi} \sigma \sqrt{c}\left(1+\frac{\beta\sigma}{\mu \sqrt{c}}\right)^{1/2},
\end{equation}
and
\begin{align}
    \frac{\sigma}{\pi}\sqrt{\frac{2}{\mu}}(C_1+C_3) & = \frac{\sqrt{2}}{\pi}\frac{\gamma\sigma}{2\sqrt{\mu}}\left(-\frac{\sigma^2-\mu}{\sigma^2}+\frac{1}{3}a\frac{\mu^2}{\sigma^4}\right)\\
    & = \frac{\sigma^2b(\beta)}{\pi\mu} \left(\frac{\mu}{\sigma^2}+\frac{1}{3}a \frac{\mu^2}{\sigma^4}-1\right),\label{eq:120}
\end{align}
where in Equation~\eqref{eq:120} also Equations~\eqref{eq:gamma} and \eqref{eq:b02} have been used. Therefore, with $\theta$ as given in Equation~\eqref{eq:theta}, we get
\begin{equation}
    \mathbb{E}[X_g] = \frac{\sqrt{2}}{\pi}\sigma\left(1+\frac{\beta\sigma}{\mu\sqrt{c}}\right)^{1/2}G_0(b(\beta))+\frac{\sqrt{2}}{\pi}\theta b(\beta)G_1(b(\beta))+O\left(\frac{1}{\sqrt{c}}\right).
\end{equation}
The expression in Equation~\eqref{eq:int_meansqrtc} is then obtained by noting that
\begin{align}
    &\left(1+\frac{\beta\sigma}{\mu\sqrt{c}}\right)^{1/2} = 1+\frac{\beta\sigma}{2\mu\sqrt{c}}+O\left(\frac{1}{c}\right),\\
    &b(\beta) = \frac{\beta}{\sqrt{2}}+O\left(\frac{1}{\sqrt{c}}\right),
\end{align}
finishing the proof of Proposition~\ref{thm:mean}.

\end{proof}

\subsection*{Proof of optimal green-time allocation using Equation~\eqref{eq:int_meansqrtc}}\label{sec:proofAlg2}

\begin{proof}
	We use the Lagrange multiplier technique to prove Theorem~\ref{thm:noCostsHigherOrder}. To start, we differentiate Equation~\eqref{eq:approach1b}
\begin{align}
\frac{\partial}{\partial \beta_j}&\sum_{i=1}^n\left( \frac{\sqrt{2}}{\pi}\left(\sigma_i\sqrt{c}+\frac{1}{2}\beta_i\frac{\sigma_{i}^2}{\mu_i}\right)G_0(b_i(\beta_i))+\frac{\theta_i\beta_i}{\pi}G_{1}\left(\frac{\beta_i}{\sqrt{2}}\right)\right)=\\
&  \frac{1}{\pi\sqrt{2}}\frac{\sigma_{j}^2}{\mu_j} G_0(b_j(\beta_j))+\frac{\sqrt{2}}{\pi}\left(\sigma_j\sqrt{c}+\frac{\beta_j\sigma_{j}^2}{2\mu_j}\right)b_{j}^\prime(\beta_j)G_{0}^\prime(b_j(\beta_j))\\
& + \frac{\theta_j}{\pi}G_1\left(\frac{\beta_j}{\sqrt{2}}\right)+\frac{\theta_j\beta_j}{\pi\sqrt{2}}G_{1}^\prime\left(\frac{\beta_j}{\sqrt{2}}\right)=\\
& \frac{\sigma_j\sqrt{c}}{\pi}G_{0}^\prime\left(\frac{\beta_j}{\sqrt{2}}\right)+\frac{1}{\pi} \left\{ \frac{\sigma_{j}^2}{\sqrt{2}\mu_j}G_{0}\left(\frac{\beta_j}{\sqrt{2}}\right)-\frac{\beta_j\sigma_{j}^2}{2\mu_j}G_{0}^\prime\left(\frac{\beta_j}{\sqrt{2}}\right)\right.\\&\left.-\frac{\beta_{j}^2\sigma_{j}^2}{2\sqrt{2}\mu_j}G_{0}^{\prime\prime}\left(\frac{\beta_j}{\sqrt{2}}\right)+\theta_jG_{1}\left(\frac{\beta_j}{\sqrt{2}}\right)+\frac{\theta_j\beta_j}{\sqrt{2}}G_{1}^\prime\left(\frac{\beta_j}{\sqrt{2}}\right)\right\}+O\left(\frac{1}{\sqrt{c}}\right),
\end{align}
where we have used/approximated that
\begin{align}
b(\beta_j) &=\frac{\beta_j}{\sqrt{2}}-\frac{\beta_{j}^2\sigma_j}{2\sqrt{2}\mu_j\sqrt{c}}+O\left(\frac{1}{c}\right),\\
 G_0(b_{j}(\beta_j)) &= G_0\left(\frac{\beta_j}{\sqrt{2}}\right)+O\left(\frac{1}{\sqrt{c}}\right),\\
 \left(\sigma_j\sqrt{c}+\frac{\beta_j\sigma_{j}^2}{2\mu_j}\right)b_{j}^\prime(\beta_j) &= \frac{\sigma_j\sqrt{c}}{\sqrt{2}}-\frac{\beta_j\sigma_{j}^2}{2\sqrt{2}\mu_j}+O\left(\frac{1}{\sqrt{c}}\right),\\
 G_{0}^\prime(b_{j}(\beta_j)) &= G_{0}^\prime\left(\frac{\beta_j}{\sqrt{2}}\right)-\frac{\beta_{j}^2\sigma_j}{2\sqrt{2}\mu_j\sqrt{c}}G_{0}^{\prime\prime}\left(\frac{\beta_{j}}{\sqrt{2}}\right)+O\left(\frac{1}{c}\right).
\end{align}

So, introducing a Lagrange multiplier $\lambda_1 \in\mathbb{R}$ and ignoring the $O$-terms, we need to solve
\begin{align}
&\lambda_1 \sigma_j\sqrt{c}=\frac{\sigma_j\sqrt{c}}{\pi}G_{0}^\prime\left(\frac{\beta_j}{\sqrt{2}}\right)+\frac{1}{\pi} \left\{ \frac{\sigma_{j}^2}{\sqrt{2}\mu_j}G_{0}\left(\frac{\beta_j}{\sqrt{2}}\right)-\frac{\beta_j\sigma_{j}^2}{2\mu_j}G_{0}^\prime\left(\frac{\beta_j}{\sqrt{2}}\right)\right.\\&\left.-\frac{\beta_{j}^2\sigma_{j}^2}{2\sqrt{2}\mu_j}G_{0}^{\prime\prime}\left(\frac{\beta_j}{\sqrt{2}}\right)+\theta_jG_{1}\left(\frac{\beta_j}{\sqrt{2}}\right)+\frac{\theta_j\beta_j}{\sqrt{2}}G_{1}^\prime\left(\frac{\beta_j}{\sqrt{2}}\right)\right\},	
\end{align}
for $j=1,...,n$. The second term on the right-hand side of the former equation is $O(1)$ and it is fair to expect that the optimal $\beta_j$ in Theorem~\ref{thm:noCostsHigherOrder} are close to the optimal solution in Theorem~\ref{thm:noCosts} in the sense that $\beta_j=\beta_*+O(1/\sqrt{c})$. Therefore, we approximate
\begin{align}
  K_j = &  \frac{\sigma_{j}^2}{\sqrt{2}\mu_j}G_{0}\left(\frac{\beta_*}{\sqrt{2}}\right)-\frac{\beta_*\sigma_{j}^2}{2\mu_j}G_{0}^\prime\left(\frac{\beta_*}{\sqrt{2}}\right)-\frac{\beta_{*}^2\sigma_{j}^2}{2\sqrt{2}\mu_j}G_{0}^{\prime\prime}\left(\frac{\beta_*}{\sqrt{2}}\right)+\\&\theta_jG_{1}\left(\frac{\beta_*}{\sqrt{2}}\right)+\frac{\theta_j\beta_*}{\sqrt{2}}G_{1}^\prime\left(\frac{\beta_*}{\sqrt{2}}\right),
\end{align}
where $K_{j}$ is as in Equation~\eqref{eq:k0i}, at the expense of an error $O(1/\sqrt{c})$. After rewriting, we then get
\begin{align}
\frac{1}{\pi}G_{0}^\prime\left(\frac{\beta_j}{\sqrt{2}}\right) = \lambda_1-\frac{K_{j}}{\pi\sigma_j\sqrt{c}}.
\end{align}
We develop, using Equation~\eqref{eq:lagrange},
\begin{align}
\frac{1}{\pi}G_{0}^\prime\left(\frac{\beta_j}{\sqrt{2}}\right) & = \frac{1}{\pi} G_{0}^\prime\left(\frac{\beta_*}{\sqrt{2}}\right)+\frac{1}{\pi\sqrt{2}}(\beta_j-\beta_*)G_{0}^{\prime\prime}\left(\frac{\beta_*}{\sqrt{2}}\right)+O\left(\frac{1}{c}\right)\\
& =  \lambda_0 + \frac{1}{\pi\sqrt{2}}(\beta_j-\beta_*)G_{0}^{\prime\prime}\left(\frac{\beta_*}{\sqrt{2}}\right)+O\left(\frac{1}{c}\right).
\end{align}
Combining the last two results, we get
\begin{align}
\frac{1}{\pi\sqrt{2}}(\beta_j-\beta_*)G_{0}^{\prime\prime}\left(\frac{\beta_*}{\sqrt{2}}\right)=\lambda_1-\lambda_0-\frac{K_{j}}{\pi\sigma_j\sqrt{c}}+O\left(\frac{1}{c}\right).
\end{align}
Deleting the $O(1/c)$ term, we find that
\begin{align}
\beta_j = \beta_* + \frac{\lambda_1-\lambda_0-\frac{K_{j}}{\pi\sigma_j\sqrt{c}}}{\frac{1}{\pi\sqrt{2}}G_{0}^{\prime\prime}\left(\frac{\beta_*}{\sqrt{2}}\right)}.
\end{align}

Using the equality constraint, we readily see that the following should hold
\begin{align}
\sum_{j=1}^n \sigma_j \left(\lambda_1-\lambda_0-\frac{K_{j}}{\pi\sigma_j\sqrt{c}}\right)=0,
\end{align}
implying that
\begin{align}
\lambda_1-\lambda_0 = \frac{1}{\pi\sqrt{c}}\frac{\sum_{j=1}^n K_{j}}{\sum_{j=1}^n \sigma_j}.
\end{align}
We thus obtain
\begin{align}
\beta_i=\beta_* +\sqrt{\frac{2}{c}}\frac{1}{G_{0}^{\prime\prime}\left(\frac{\beta_*}{\sqrt{2}}\right)}\left(\frac{\sum_{j=1}^n K_{j}}{\sum_{j=1}^n\sigma_j}-\frac{K_{i}}{\sigma_i}\right),
\end{align}
concluding the proof.
\end{proof}
%
%
%
%
%

%
%

\section{Comparison with Webster}\label{app:webster}

In this appendix, we compare our approach to the classical approach by Webster \cite{webster1958traffic}. In his seminal paper, Webster gives an approximation formula for the mean delay at traffic intersections with static signaling. 
He also discusses briefly how to optimally allocate capacity to the different flows. One of the key arguments, which we tacitly adopt in our paper, is that it suffices to consider only the busiest approach in each phase, i.e. the approach with the highest ratio of flow to saturation flow. Webster's approximation for the mean delay is given by
\begin{equation}\label{eqn:webster}
\mathbb{E}[D] \approx \frac{(c-g)^2}{2c(1-\rho)}  + \frac{\rho c^2}{2 g ( g - \mu c)}
- 0.65 \left( \frac{c}{\mu^2} \right)^{1/3}
\left( \frac{\mu c}{g} \right)^{2+5 g/c} .
\end{equation}
Note that Webster's delay formula has no term that accounts for the variability in the arrival process: It seems that this approximation implicitly assumes Poisson arrivals. 
The unit of delay is the length of a time slot, which is the interdeparture time of two consecutive delayed vehicles (approximately two seconds, on average).

The \emph{exact} delay of a fixed-cycle traffic light queue has been determined in \cite{van2006delay}. It turns out that the mean delay $\E[D]$ can be easily expressed in terms of the mean overflow queue $\E[X_g]$:
\begin{equation}\label{eqn:meandelay}
\mathbb{E}[D] = \frac{c-g}{2c\mu(1-\mu)}\left( \frac{\sigma^2}{1-\mu}+(c-g)\mu+2\E[X_g]\right).
\end{equation}

As Webster notices, it is generally suggested (in particular in engineering handbooks) that the optimal green time allocation is such that the green period lengths are in proportion to their corresponding ratios of flow. Indeed, with a cycle of length $c$ and a total all-red time of $r_T$, we get the following capacity allocation rule:
\begin{equation}\label{eqn:websterCapacityAllocation}
g_i = \frac{\mu_i}{\sum_{j=1}^n\mu_j}(c-r_T).
\end{equation}
In \cite{webster1958traffic}, Webster applies this rule in all of his numerical examples. We revisit the four-lane example from Subsection \ref{sec:heavy_four} and compare our results with those obtained by applying Webster's approach. Since Webster does not consider weights in his objective functions, we take the setting where $d_1=\dots=d_4=1$. Distributing the available green time according to \eqref{eqn:websterCapacityAllocation} 
results in the green times tabulated in Table~\ref{tbl:websterCapacity}. Note that the green times in lanes 1 and 2 are the same, due to the fact that 
Webster does not take into account the differences in variability of the arrival processes. Moreover, the vehicle-to-capacity ratio $\rho$ is the same for all lanes. The differences with Table~\ref{t:newNumExample2EqualWeight} appear to be relatively small, but the impact on the mean delays can be quite substantial, in particular when the intersection is operating close to its maximum capacity as in this example. Since Webster focuses on the mean delay, rather than on the overflow queue, we use Equation~\eqref{eqn:meandelay} to convert the values in (the upper part of) Table~\ref{t:newNumExample2EXg} to exact and approximated mean delays. The result is shown in (the upper part of) Table~\ref{t:newNumExample2EqualWeighDelays}. The last column of this table contains the mean delay of an arbitrary vehicle, which is the weighted average (with weights $\mu_i/\sum_{j=1}^n \mu_j$) of the delays per lane.

\begin{table}[h!]
	\centering
\small \begin{tabular}{c|cccc|c}
	$c$ & $g_1$ & $g_2$ & $g_3$ & $g_4$ & $\rho$        \\\hline
	30  &   9.375 & 9.375 & 3.125 & 3.125 & 0.960 \\
	50  &   16.875 & 16.875 & 5.625 & 5.625 & 0.889 \\
	100 &   35.625 & 35.625 & 11.875 & 11.875 & 0.842 \\
	200 &   73.125 & 73.125 & 24.375 & 24.375 & 0.821 \\
	500 &   185.625 & 185.625 & 61.875 & 61.875 & 0.808 \\
	\end{tabular}
	\caption{{\small Optimal green times according to capacity allocation rule~\eqref{eqn:websterCapacityAllocation}, for the four-lane intersection with $r_T=5$ and $d_i=1$ for various values of $c$.}}
\label{tbl:websterCapacity}
\end{table}
We now compare the performance of dimensioning rule \eqref{eq:approach2b} with Webster's settings. For this purpose, we compute the exact mean delay for an intersection with green times as in Table~\ref{tbl:websterCapacity}. Furthermore, we also approximate the mean delay using Webster's formula \eqref{eqn:webster}. The results are given in the lower part of Table~\ref{t:newNumExample2EqualWeighDelays}, again adding a column with the mean delay of an arbitrary vehicle. It becomes clear that our dimensioning rule ensures more fairness by improved balancing of the mean waiting times: flows 3 and 4 have the highest waiting times, but these are much smaller compared to Webster's schedule. Regarding the mean waiting time of an arbitrary vehicle, our scheme performs significantly better, in particular in the (more realistic) cases where the cycle times are less than 200 time units (approximately 400 seconds). A final conclusion from Table~\ref{t:newNumExample2EqualWeighDelays} is that Webster's approximation for the mean delay (Eq. \eqref{eqn:webster}) is not as accurate as our refined approximation \eqref{eq:int_meansqrtc}, but still remarkably accurate for traffic flows with Poisson arrivals (flow 2, in our example). 
\begin{table}[h!]
	\centering
{\small \begin{tabular}{c|cc|cc|cc|c|}
\multicolumn{8}{c}{Capacity allocation rule \eqref{eq:approach2b}}\\
\hline
	$c$ & $\E[D_1]$ & Eq.~\eqref{eq:int_meansqrtc} & $\E[D_2]$ & Eq.~\eqref{eq:int_meansqrtc} & $\E[D_3], \E[D_4]$ & Eq.~\eqref{eq:int_meansqrtc} & $\E[D]$ \\\hline
	30  &   81.697 & 80.828 & 72.996 & 71.963 & 235.232 & 236.339 & 116.818 \\
	50  &   35.031 & 35.028 & 32.666 & 32.665 & 83.112 & 85.993 & 46.164 \\
	100 &   39.872 & 39.962 & 39.144 & 39.154 & 71.492 & 73.852 & 47.504 \\
	200 &   65.872 & 65.950 & 66.210 & 66.253 & 98.059 & 99.753 & 74.045 \\
	500 &   153.307 & 153.349 & 155.766 & 155.794 & 207.683 & 208.695 & 167.823 \\
\hline
\multicolumn{8}{c}{Capacity allocation rule \eqref{eqn:websterCapacityAllocation}}\\
\hline
	$c$ & $\E[D_1]$ & Eq.~\eqref{eqn:webster} & $\E[D_2]$ & Eq.~\eqref{eqn:webster} & $\E[D_3], \E[D_4]$ & Eq.~\eqref{eqn:webster} & $\E[D]$ \\\hline
	30  &   57.380 & 44.631 & 45.974 & 44.631 & 484.747 & 120.117 & 159.944 \\
	50  &   26.768 & 24.066 & 23.571 & 24.066 & 166.618 & 49.216 & 60.531 \\
	100 &   33.983 & 33.571 & 32.458 & 33.571 & 125.701 & 56.633 & 56.341 \\
	200 &   59.490 & 59.764 & 58.718 & 59.764 & 143.194 & 93.903 & 80.127 \\
	500 &   142.130 & 141.797 & 141.858 & 141.797 & 247.776 & 216.589 & 168.440 \\
	\hline
	\end{tabular}
	\caption{{\small 
Mean delays for a four-lane example with geometric arrivals with mean 0.3 in lane 1, Poisson arrivals with mean 0.3 in lane 2, negative binomial arrivals with mean 0.1 and variance 0.4 in lanes 3 and 4, with $r_T=5$ and $d_i=1$ for various values of $c$.}}
	\label{t:newNumExample2EqualWeighDelays}}
\end{table}

\end{document}